\documentclass[12pt]{article}

\usepackage{amsmath,amssymb,amsthm,mathtools}
\usepackage[a4paper,margin=30mm]{geometry}
\usepackage[hidelinks]{hyperref}

\numberwithin{equation}{section}

\newtheorem{theorem}{Theorem}[section]
\newtheorem{proposition}[theorem]{Proposition}
\newtheorem{lemma}[theorem]{Lemma}
\newtheorem{corollary}[theorem]{Corollary}
\theoremstyle{definition}
\newtheorem{definition}[theorem]{Definition}
\theoremstyle{remark}
\newtheorem{remark}[theorem]{Remark}

\DeclareMathOperator{\Aut}{Aut}
\DeclareMathOperator{\Cay}{Cay}
\DeclareMathOperator{\End}{End}

\DeclareMathOperator{\GammaO}{\Gamma O}
\DeclareMathOperator{\Orth}{O}
\DeclareMathOperator{\PG}{PG}
\DeclareMathOperator{\rad}{rad}
\DeclareMathOperator{\Sym}{Sym}
\DeclareMathOperator{\Tr}{Tr}

\newcommand{\F}{\mathbb F}
\newcommand{\calO}{\mathcal O}
\newcommand{\calQ}{\mathcal Q}
\newcommand{\eps}{\varepsilon}
\newcommand{\restr}[2]{\left.#1\right|_{#2}}

\title{First Subconstituents of Orthogonal Graphs in Even Characteristic:\\
Automorphisms, Cliques, Cores, and Binary Triple Transitivity}

\author{Kai Zhou$^{1}$ and Hongfeng Wu$^{2}$\footnote{Corresponding author.}
\setcounter{footnote}{-1}
\footnote{E-Mail addresses:
kzhou@cslg.edu.cn (K. Zhou), whfmath@gmail.com (H. Wu)}
\\
{1.~School of Mathematics and Statistics, Suzhou University of Technology,}\\
{Changshu 215500, P. R. China}
\\
{2.~College of Science, North China University of Technology, Beijing, China}}

\date{}

\begin{document}

\maketitle

\begin{abstract}
In even characteristic, the first subconstituent of an orthogonal graph at a singular point is shown to be the complement of an affine polar graph. 
This identification yields the full automorphism group as an affine semisimilarity group and provides explicit extensions to the ambient graph.
Maximum independent sets are affine cosets of maximal totally singular subspaces, giving the independence and chromatic numbers.
In plus type, Delsarte cliques correspond bijectively to ovoids through the base point, leading to a sharp rank threshold.
Over $\mathbb{F}_2$, we determine the exact maximum clique number in every dimension, prove that all maximum cliques form a single orbit, and obtain the core of every first subconstituent.
A uniform Witt‑extension argument gives a short proof of triple transitivity for the binary local graphs.

\medskip
\noindent\textbf{Keywords.}
Orthogonal graph; first subconstituent; affine polar graph; maximum clique;
ovoid; graph core; Terwilliger algebra; semisimilarity.\\
\textbf{Mathematics Subject Classification}: 05E30,\ 05C25,\ 51E20,\ 51A50
\end{abstract}

\section{Introduction}

Orthogonal graphs in characteristic two are naturally defined from a
quadratic form rather than from a symmetric bilinear form.  This distinction
often makes a direct coordinate analysis long: singularity is quadratic,
whereas adjacency is governed by the associated alternating polar form.
The two structures separate particularly cleanly in a neighbourhood of a
singular point.

Let $p$ be a singular vector and choose a singular vector $f$ with
$B(p,f)=1$.  The hyperbolic plane $\langle p,f\rangle$ splits off, and its
orthogonal complement $W$ is again a nondegenerate quadratic space.  Every
neighbour of $[p]$ has a unique normalized representative
\[
  v_x=f+x+Q_W(x)p \qquad (x\in W).
\]
The elementary identity
\[
  B(v_x,v_y)=Q_W(x-y)
\]
then identifies the entire first subconstituent with the complement of the
affine polar graph on $W$.  Thus a projective local problem becomes an
affine Cayley-graph problem in one step.

Affine polar graphs go back to the geometry developed by Cohen and Shult
\cite{CohenShult}.  Their exact automorphism groups are also the affine-polar
case of the exact $2$-closure theorem for rank-three groups
\cite{Skresanov}; see also the recent consolidated treatment
\cite{GuoVasilevWang}.  Combining that theorem with the affine chart yields
the automorphism group without reconstructing field operations from many
families of specially chosen vertices.

The same reduction opens two further directions.  First, cliques in the
local graph are anisotropic-difference sets, while independent sets are
affine totally singular flats.  Delsarte cliques in plus type are therefore
the affine shadows of ovoids.  This connects the local graph to the sharp
separation results and ovoid obstructions in
\cite{BambergEtAlSeparating,BambergEtAlTactical}.  Second, the theorem of
Roberson that every primitive strongly regular graph is a pseudocore
\cite{Roberson} converts the clique and coloring information into an exact
endomorphism theorem.

There is also a recent development concerning Terwilliger algebras.
Herman, Maleki, and Razafimahatratra isolated the binary affine polar graphs
as one of the last open families in their classification of triply
transitive strongly regular graphs \cite{HermanEtAl}.  Li and Zou have since
proved the conjecture \cite{LiZou}.  The affine chart below gives a short
uniform proof based only on the recovery of a two-vector Gram matrix and
Witt extension; it simultaneously transfers the conclusion to the local
graphs, which are the complementary relation.

The odd-dimensional orthogonal graph in characteristic two is isomorphic to
the corresponding symplectic graph \cite{WanZhou}; its automorphisms and
subconstituents have been considered in \cite{TangWan,LiWang,GuWan}.  We
therefore concentrate on the two nondegenerate even-dimensional orthogonal
types.  In the notation
$O(2\nu+\delta,q)$ used in \cite{WanZhou}, these are $\delta=0$ and
$\delta=2$.

Our principal result may be stated as follows.  The notation
$\GammaO_{2m}^{\eps}(q)$ denotes the full semisimilarity group of the quadratic
space; it is defined precisely in Section~\ref{sec:preliminaries}.

\begin{theorem}\label{thm:main-intro}
Let $q=2^h$, let $(V,Q)$ be a nondegenerate quadratic space of dimension
$2m+2$ and type $\eps\in\{+1,-1\}$, and let $[p]$ be a singular point.
Write $\Gamma_1(p)$ for the graph induced by $\calO(V,Q)$ on the neighbours
of $[p]$.
\begin{enumerate}
\item If $m\ge2$, then
\[
  \Aut\Gamma_1(p)\cong W\rtimes\GammaO_{2m}^{\eps}(q),
\]
where $W$ is the additive group of a $2m$-dimensional quadratic space of
type $\eps$.
\item If $m=1$ and $\eps=+1$, then
\[
  \Aut\Gamma_1(p)\cong \Sym(q)\wr\Sym(2)
\]
in its product action.
\item If $m=1$ and $\eps=-1$, then $\Gamma_1(p)\cong K_{q^2}$ and
$\Aut\Gamma_1(p)\cong\Sym(q^2)$.
\end{enumerate}
In every case, every automorphism of $\Gamma_1(p)$ is the restriction of an
automorphism of the ambient orthogonal graph that fixes $[p]$.
\end{theorem}

For $m\ge2$ the group order is
\begin{equation}\label{eq:intro-order}
 |\Aut\Gamma_1(p)|
 =2h(q-1)q^{m(m+1)}(q^m-\eps)
   \prod_{i=1}^{m-1}(q^{2i}-1).
\end{equation}
The exceptional hyperbolic and elliptic orders are $2(q!)^2$ and
$(q^2)!$, respectively.

The extremal, core, and binary consequences can be summarized as follows.
Here $\alpha$, $\omega$, and $\chi$ denote the independence, clique, and
chromatic numbers, respectively, and
\begin{equation}\label{eq:eta-intro}
  \eta_m=(-1)^{m(m+1)/2}.
\end{equation}

\begin{theorem}\label{thm:further-intro}
Let $G=\Gamma_1(p)$.
\begin{enumerate}
\item If $\eps=+1$, then $\alpha(G)=\chi(G)=q^m$.  If $\eps=-1$, then
\[
  \alpha(G)=q^{m-1},\qquad \chi(G)=q^{m+1}.
\]
All maximum independent sets are affine cosets of maximal totally singular
subspaces.
\item If $q=2$, then
\[
 \omega(G)=
 \begin{cases}
  2m+2,&\eps=\eta_m\text{ and }m\text{ is odd},\\
  2m+1,&\eps=\eta_m\text{ and }m\text{ is even},\\
  2m,&\eps=-\eta_m.
 \end{cases}
\]
Moreover, $\Aut(G)$ is transitive on the maximum cliques.
\item The core of $G$ is
\[
 \operatorname{core}(G)\cong
 \begin{cases}
  K_{q^m},&\eps=+1\text{ and }1\le m\le3,\\
  G,&\eps=+1\text{ and }m\ge4,\\
  G,&\eps=-1.
 \end{cases}
\]
Thus, in the last two cases every endomorphism is an automorphism.
\item If $q=2$ and $m\ge2$, then $G$ is triply transitive in the sense that
its basic Terwilliger space, Terwilliger algebra, and vertex-stabilizer
centralizer algebra coincide.
\end{enumerate}
\end{theorem}

Besides shortening the proof, this formulation makes the rank conditions
transparent.  In particular, the elliptic case corresponding to
$(\nu,\delta)=(2,2)$ has $m=2$ and is covered by the uniform affine-polar
theorem.  The only genuinely exceptional case in the range $\nu\ge2$ is
the hyperbolic case $(\nu,\delta)=(2,0)$.

\section{Quadratic and affine polar graphs}\label{sec:preliminaries}

Throughout, $q=2^h$ and $\F_q$ is the field with $q$ elements.  Let $V$ be
a finite-dimensional $\F_q$-vector space and let $Q:V\to\F_q$ be a
quadratic form.  Its polar form is
\begin{equation}\label{eq:polar}
  B(u,v)=Q(u+v)-Q(u)-Q(v).
\end{equation}
Because the characteristic is two, $B$ is alternating.  We call $Q$
nondegenerate when $B$ is nondegenerate.  Consequently a nondegenerate
quadratic space has even dimension.

Up to isometry there are two nondegenerate quadratic forms in dimension
$2r$.  We label their types by $\eps=+1$ (hyperbolic, or plus type) and
$\eps=-1$ (elliptic, or minus type).  Convenient representatives are
\begin{align}
 Q_{2r}^{+}(x_1,\ldots,x_r,y_1,\ldots,y_r)
   &=\sum_{i=1}^{r}x_i y_i,\label{eq:plus-form}\\
 Q_{2r}^{-}(x_1,\ldots,x_{r-1},y_1,\ldots,y_{r-1},u,v)
   &=\sum_{i=1}^{r-1}x_i y_i+\alpha u^2+uv+\alpha v^2,
   \label{eq:minus-form}
\end{align}
where $\Tr_{\F_q/\F_2}(\alpha)=1$.  The trace condition is equivalent to
$\alpha\notin\{z^2+z:z\in\F_q\}$ and makes the final binary form
anisotropic.

\begin{definition}
Let $(V,Q)$ be nondegenerate.  The graph $\calO(V,Q)$ has as its vertices
the singular one-dimensional subspaces
\[
  \{[v]\in\PG(V):Q(v)=0\}.
\]
Two distinct vertices $[u]$ and $[v]$ are adjacent when $B(u,v)\ne0$.
This condition is independent of the chosen nonzero representatives.
For a vertex $[p]$, its \emph{first subconstituent} $\Gamma_1(p)$ is the
subgraph induced on the neighbours of $[p]$.
\end{definition}

This adjacency convention is the complement, on the singular points, of
the usual polar-space collinearity relation.  It agrees with the convention
for the orthogonal graphs $O(2\nu+\delta,q)$ in \cite{WanZhou}.

\begin{definition}
Let $(W,Q_W)$ be a nondegenerate quadratic space.  Its affine polar graph is
\[
 VO(W,Q_W)=\Cay\bigl(W,S_0\bigr),\qquad
 S_0=\{z\in W\setminus\{0\}:Q_W(z)=0\}.
\]
Thus distinct $x,y\in W$ are adjacent precisely when $Q_W(x-y)=0$.
We write $VO_{2m}^{\eps}(q)$ when $\dim W=2m$ and $W$ has type $\eps$.
\end{definition}

Let $\Aut(\F_q)=\langle z\mapsto z^2\rangle$, a cyclic group of order
$h$.  The semisimilarity group of $(W,Q_W)$ is $\GammaO(W,Q_W)$
\begin{equation}\label{eq:GammaO-definition}
\begin{split}
 =\{A:W\to W:\;&A\text{ is an invertible }
 \sigma\text{-semilinear map for some }\sigma\in\Aut(\F_q),\\
 &Q_W(Ax)=\lambda\,\sigma(Q_W(x))
 \text{ for some }\lambda\in\F_q^*,\ \forall x\in W\}.
\end{split}
\end{equation}
The scalar $\lambda$ is the multiplier. We write $\GammaO_{2m}^{\eps}(q)$ when $\dim W=2m$ and $W$ has type $\eps$. The corresponding affine group is
\[
 A\GammaO(W,Q_W)=W\rtimes\GammaO(W,Q_W),
\]
the semidirect product of the additive group of $W$ and the semisimilarity group. The product is given by
\[
(b_1,A_1)(b_2,A_2)=(b_1+A_1b_2,\;A_1A_2),
\]
and the action on $W$ is
\[
x\longmapsto Ax+b\qquad (A\in\Gamma O(W,Q_W),\;b\in W).
\]

The modern rigidity input used below is the following exact description of the automorphism group of affine polar graphs.

\begin{theorem}\label{thm:affine-polar-aut}
Let $(W,Q_W)$ be nondegenerate of dimension $2m\ge4$.  Then
\[
 \Aut\bigl(VO(W,Q_W)\bigr)=A\GammaO(W,Q_W).
\]
The same equality holds for the complement of $VO(W,Q_W)$.
\end{theorem}

\begin{proof}
The affine polar graphs occur among the affine rank-three graphs.  Their
exact automorphism groups, equivalently the exact $2$-closures of the
associated rank-three groups, are determined in \cite{Skresanov}; the
affine-polar case is restated in \cite{GuoVasilevWang}.  See also the
rank-three tables in \cite{BrouwerVanMaldeghem}.  Passing to a graph
complement does not change its automorphism group.
\end{proof}

For geometric intuition, every coset of a totally singular subspace is a
clique in the affine polar graph.  Conversely, after translating a clique
through the origin, all pairwise differences are singular; polarization
then shows that its linear span is totally singular.  The theorem says, in
particular, that the global incidence among these affine singular flats
recovers both the affine structure and the quadratic form up to field
automorphism and multiplier.  This is the rigidity that a direct coordinate
proof must reconstruct one vertex family at a time.

\section{The affine chart at a singular point}\label{sec:chart}

We now give the key reduction.  It uses only Witt decomposition and one
polarization identity.

\begin{proposition}\label{prop:chart}
Let $(V,Q)$ be a nondegenerate quadratic space of dimension $2m+2$ and type
$\eps$, and let $p\in V\setminus\{0\}$ be singular.  Choose a singular
vector $f$ with $B(p,f)=1$ and put
\[
 H=\langle p,f\rangle,\qquad W=H^\perp.
\]
Then $(W,Q_W)$ is nondegenerate of dimension $2m$ and type $\eps$.  The map
\begin{equation}\label{eq:affine-chart}
 \iota:W\longrightarrow V(\Gamma_1(p)),\qquad
 x\longmapsto[v_x],\quad v_x=f+x+Q_W(x)p,
\end{equation}
is a bijection.  Moreover,
\begin{equation}\label{eq:key-identity}
 B(v_x,v_y)=Q_W(x-y)\qquad(x,y\in W).
\end{equation}
Consequently
\begin{equation}\label{eq:local-cayley}
 \Gamma_1(p)\cong
 \Cay\bigl(W,\{z\in W:Q_W(z)\ne0\}\bigr)
 =\overline{VO(W,Q_W)}.
\end{equation}
\end{proposition}

\begin{proof}
The plane $H$ is hyperbolic, so $V=H\perp W$.  Removing a hyperbolic plane
preserves the orthogonal type and leaves a nondegenerate space of dimension
$2m$.

Let $[u]$ be a vertex of $\Gamma_1(p)$, i.e.\ $[u]\neq[p]$ and $B(p,u)\neq0$.  Since $B(p,u)\ne0$, there is a unique
representative of this point satisfying $B(p,u)=1$.  Write this normalized vector in the decomposition $V = \langle p,f\rangle \perp W$ as
\[
u = \alpha p + \beta f + x \qquad (\alpha,\beta\in\mathbb{F}_q,\; x\in W).
\]
Using $B(p,f)=1$, $B(p,p)=0$ and $x\perp p$, we obtain
\[
1 = B(p,u) = \beta B(p,f) = \beta,
\]
so $\beta=1$.  Now the singularity of $u$ gives
\[
0 = Q(u) =B(\alpha p + f,x)+ Q(\alpha p + f) + Q_W(x) = \alpha + Q_W(x),
\]
because $Q(\alpha p+f) = \alpha^2 Q(p) + Q(f) + \alpha B(p,f) = \alpha$ (recall $Q(p)=Q(f)=0$ and $B(p,f)=1$).  Hence $\alpha = Q_W(x)$, and
\[
u = f + x + Q_W(x)p =: v_x.
\]
Thus every neighbour of $[p]$ can be written uniquely as $[v_x]$ for some $x\in W$. Conversely, for any $x\in W$ one checks directly that
$$
Q(v_x) = Q(f + x + Q_W(x)p) =B(f +  Q_W(x)p,x)+Q(f +  Q_W(x)p)+Q_W(x)
$$
$$
=B(f,Q_W(x)p)+Q(f)+Q(Q_W(x)p)+Q_W(x)=Q_W(x)+Q_W(x)=0
$$
and $B(p,v_x) = B(p,f)=1$, so $[v_x]$ is indeed adjacent to $[p]$.  This proves the bijectivity of \eqref{eq:affine-chart}.

Using $B(p,f)=1$, the orthogonality of $W$ to $H$, and
characteristic two, we obtain
\begin{align*}
 B(v_x,v_y)
 &=B(x,y)+Q_W(x)+Q_W(y)\\
 &=Q_W(x+y)=Q_W(x-y).
\end{align*}
Thus two distinct chart points are adjacent in $\Gamma_1(p)$ exactly when
their difference is nonsingular, proving \eqref{eq:local-cayley}.
\end{proof}

\begin{remark}
The normalization $B(p,u)=1$ is what turns a projective neighbourhood into
an affine space.  No choice of a square root is involved.  Although the
Frobenius map $z\mapsto z^2$ is bijective on $\F_q$, that fact is not needed
for the chart itself.
\end{remark}

\section{Parameters of the first subconstituent}\label{sec:parameters}

The chart of Proposition $\ref{prop:chart}$ also makes the strongly regular parameters immediate.  We record
them both as a check and for comparison with earlier computations of
subconstituents \cite{ZhouGuWan}.

For $m\ge2$, the affine polar graph $VO_{2m}^{\eps}(q)$ has parameters
\begin{align}
 v_0&=q^{2m},\label{eq:vo-v}\\
 k_0&=(q^m-\eps)(q^{m-1}+\eps),\label{eq:vo-k}\\
 \lambda_0&=q(q^{m-1}-\eps)(q^{m-2}+\eps)+q-2,\label{eq:vo-lambda}\\
 \mu_0&=q^{m-1}(q^{m-1}+\eps).\label{eq:vo-mu}
\end{align}
Applying the standard complement transformation \cite{BCN,GodsilRoyle}
\[
 (v,k,\lambda,\mu)\longmapsto
 (v,v-k-1,v-2k+\mu-2,v-2k+\lambda)
\]
and simplifying gives the following result.

\begin{proposition}\label{prop:local-parameters}
If $m\ge2$, then $\Gamma_1(p)$ is strongly regular with parameters
\begin{align}
 v&=q^{2m},\label{eq:local-v}\\
 k&=(q-1)q^{m-1}(q^m-\eps),\label{eq:local-k}\\
 \mu&=(q-1)q^{m-1}\bigl((q-1)q^{m-1}-\eps\bigr),
 \label{eq:local-mu}\\
 \lambda&=\mu-\eps(q-2)q^{m-1}.
 \label{eq:local-lambda}
\end{align}
\end{proposition}
In \cite[Theorem~3.9]{ZhouGuWan}, the parameters $a$ and $c$ correspond to $\lambda$ and $\mu$, respectively, in the notation of Proposition~\ref{prop:local-parameters}.

When $m=1$ and $\eps=+1$, the coordinates in
\eqref{eq:plus-form} give $Q_W(s,t)=st$.  Hence
\begin{equation}\label{eq:grid-adjacency}
 (s,t)\sim(s',t')\quad\Longleftrightarrow\quad
 s\ne s'\text{ and }t\ne t'.
\end{equation}
This graph is the complement of the Hamming graph $H(2,q)$ and has
parameters
\begin{equation}\label{eq:rank-one-plus-parameters}
 \bigl(q^2,(q-1)^2,(q-2)^2,(q-1)(q-2)\bigr).
\end{equation}
For $m=1$ and $\eps=-1$, the form $Q_W$ is anisotropic, so every nonzero
difference is nonsingular and $\Gamma_1(p)=K_{q^2}$.

\section{The full automorphism group}\label{sec:automorphisms}

We first treat the uniform range.

\begin{theorem}\label{thm:local-aut}
Let $m\ge2$.  Under the chart \eqref{eq:affine-chart},
\[
 \Aut\Gamma_1(p)=A\GammaO(W,Q_W)
 =W\rtimes\GammaO(W,Q_W).
\]
\end{theorem}

\begin{proof}
By Proposition~\ref{prop:chart}, the local graph $\Gamma_1(p)$ is the complement of
$VO(W,Q_W)$.  Theorem~\ref{thm:affine-polar-aut} gives its full
automorphism group.
\end{proof}

\begin{corollary}\label{cor:order}
For $m\ge2$,
\[
 |\Aut\Gamma_1(p)|
 =2h(q-1)q^{m(m+1)}(q^m-\eps)
   \prod_{i=1}^{m-1}(q^{2i}-1).
\]
\end{corollary}
\begin{proof}
By Theorem~\ref{thm:local-aut}, $\Aut\Gamma_1(p) \cong A\GammaO(W,Q_W) = W \rtimes \GammaO(W,Q_W)$.  Hence
\[
|\Aut\Gamma_1(p)| = |W| \cdot |\GammaO_{2m}^{\eps}(q)| = q^{2m} \cdot |\GammaO_{2m}^{\eps}(q)|.
\]

We now compute $|\GammaO_{2m}^{\eps}(q)|$.  Every semisimilarity $A \in \GammaO_{2m}^{\eps}(q)$ consists of a $\sigma$-semilinear map ($\sigma \in \Aut(\F_q)$) and a multiplier $\lambda \in \F_q^*$ satisfying $Q_W(Ax) = \lambda \sigma(Q_W(x)),\ \forall x\in W$.  The map
\[
\GammaO_{2m}^{\eps}(q) \longrightarrow \Aut(\F_q) \times \F_q^*, \qquad
A \longmapsto (\sigma, \lambda),
\]
is a group homomorphism.  Its kernel consists of those $A$ with $\sigma =\mathrm{id}$ and $\lambda = 1$, i.e.\ $A$ is linear and preserves $Q_W$, which is precisely the orthogonal group $\Orth_{2m}^{\eps}(q)$.  Thus we have an exact sequence
\[
1 \longrightarrow \Orth_{2m}^{\eps}(q) \longrightarrow \GammaO_{2m}^{\eps}(q) \longrightarrow \Aut(\F_q) \times \F_q^*.
\]

The image of this homomorphism is all of $\Aut(\F_q) \times \F_q^*$.  Surjectivity onto $\F_q^*$ is seen by taking $\sigma = \mathrm{id}$ and $A$ to be scalar multiplication by $c \in \F_q^*$: then $Q_W(cx) = c^2 Q_W(x)$, and since squaring is an automorphism of $\F_q$ in characteristic two, every element of $\F_q^*$ is a square, hence every multiplier is realized.  Surjectivity onto $\Aut(\F_q)$ for the standard plus type form $\sum x_i y_i$ follows from a similar, and simpler, argument as in the minus type case. For the minus type, we show that for every field automorphism $\sigma\in\Aut(\F_q)$ there exists $A\in\GammaO_{2m}^{\eps}(q)$ with multiplier $\lambda=1$ and associated automorphism $\sigma$, i.e.\ $A$ is $\sigma$-semilinear and $Q_W(Ax)=\sigma(Q_W(x))$ for all $x\in W$.

The space $W$ splits as an orthogonal direct sum $W = H_{2m-2}\perp E$, where $H_{2m-2}$ is hyperbolic and $E$ is an anisotropic plane. Consider the twisted space $W^{(\sigma)}$: as an additive group it is $W$, but the scalar multiplication is defined by $c\cdot_\sigma x = \sigma^{-1}(c)x$. Equip $W^{(\sigma)}$ with the quadratic form
\[
Q_W^{(\sigma)}(x) = \sigma(Q_W(x)).
\]
A direct verification shows that $Q_W^{(\sigma)}$ is a nondegenerate quadratic form on $W^{(\sigma)}$.  Indeed, its polar form is $B^{(\sigma)}(x,y) = \sigma(B(x,y))$, and $\sigma$ preserves nondegeneracy.  The identity map $\mathrm{id}_\sigma \colon W \to W^{(\sigma)}$ is then a $\sigma$-semilinear isomorphism and satisfies
\[
Q_W^{(\sigma)}(\mathrm{id}_\sigma(x)) = \sigma(Q_W(x)).
\]

We now check that $Q_W^{(\sigma)}$ is again of minus type.  Under the orthogonal decomposition $W^{(\sigma)} = H_{2m-2}^{(\sigma)} \perp E^{(\sigma)}$, the hyperbolic part remains hyperbolic (the standard form $\sum x_i y_i$ is unchanged after twisting, up to an obvious isometry). For the anisotropic plane $E$, choose a basis $(e_1, e_2)$ such that
\[
Q_E(se_1+te_2) = \alpha s^2 + st + \alpha t^2,\ \forall s,t\in \F_q,
\]
where $\Tr_{\F_q/\F_2}(\alpha) = 1.$ Let a vector in $E^{(\sigma)}$ be written as $u \cdot_{\sigma} e_1 + v \cdot_{\sigma} e_2$, where $(e_1, e_2)$ is the same basis.
In the original space this corresponds to $x = \sigma^{-1}(u) e_1 + \sigma^{-1}(v) e_2$.
Then
\[
\begin{aligned}
Q_E^{(\sigma)}(u \cdot_{\sigma} e_1 + v \cdot_{\sigma} e_2) &= Q_E^{(\sigma)}(x) \\
&= \sigma\bigl(Q_E(\sigma^{-1}(u)e_1+\sigma^{-1}(v)e_2)\bigr)\\
&= \sigma\Bigl( \alpha (\sigma^{-1}(u))^2 + \sigma^{-1}(u)\sigma^{-1}(v) + \alpha (\sigma^{-1}(v))^2 \Bigr) \\
&= \sigma(\alpha) u^2 + uv + \sigma(\alpha) v^2.
\end{aligned}
\]  
Since
\[
\Tr_{\F_q/\F_2}(\sigma(\alpha)) = \sigma(\Tr_{\F_q/\F_2}(\alpha)) = \sigma(1) = 1,
\]
the trace condition is preserved, so $Q_E^{(\sigma)}$ is again anisotropic.  Hence $Q_W^{(\sigma)}$ is a nondegenerate quadratic space of dimension $2m$ and minus type.

Over a finite field, any two nondegenerate quadratic spaces of the same dimension and the same type are linearly isometric.  Therefore there exists a linear isometry
\[
\Phi \colon (W^{(\sigma)}, Q_W^{(\sigma)}) \longrightarrow (W, Q_W).
\]
(Here ``linear'' means with respect to the scalar multiplication $\cdot_\sigma$ on the domain and the original scalar multiplication on the codomain, i.e.\ $\Phi(c\cdot_\sigma x) = c\,\Phi(x)$.) Equivalently, this follows from the classification of affine quadrics in even characteristic(\cite[Theorem~1.36]{Wan}), since two homogeneous nondegenerate quadrics of the same type can be transformed into the same standard form by an affine transformation, which must be linear because the origin is preserved.

Now define $A = \Phi \circ \mathrm{id}_\sigma \colon W \to W$.  Then $A$ is $\sigma$-semilinear because:
\[
A(cx) = \Phi(\mathrm{id}_\sigma(cx)) = \Phi(\sigma(c)\cdot_\sigma \mathrm{id}_\sigma(x)) = \sigma(c)\,\Phi(\mathrm{id}_\sigma(x)) = \sigma(c) A(x).
\]
Moreover,
\[
Q_W(Ax) = Q_W(\Phi(\mathrm{id}_\sigma(x))) = Q_W^{(\sigma)}(\mathrm{id}_\sigma(x)) = \sigma(Q_W(x)).
\]
Thus $A \in \GammaO_{2m}^{\eps}(q)$ with multiplier $\lambda=1$ and the prescribed field automorphism $\sigma$.  This completes the proof of surjectivity for the minus type.

Therefore
\[
|\GammaO_{2m}^{\eps}(q)| = |\Aut(\F_q)| \cdot |\F_q^*| \cdot |\Orth_{2m}^{\eps}(q)|
= h \cdot (q-1) \cdot |\Orth_{2m}^{\eps}(q)|.
\]

The order of the even-characteristic orthogonal group is well known (\cite[Theorem~7.23]{Wan}):
\[
|\Orth_{2m}^{\eps}(q)| = 2 q^{m(m-1)} (q^m - \eps) \prod_{i=1}^{m-1} (q^{2i} - 1).
\]

Putting everything together,
\[
\begin{aligned}
|\Aut\Gamma_1(p)|
&= q^{2m} \cdot h (q-1) \cdot 2 q^{m(m-1)} (q^m - \eps) \prod_{i=1}^{m-1} (q^{2i} - 1)\\
&= 2 h (q-1) q^{m(m+1)} (q^m - \eps) \prod_{i=1}^{m-1} (q^{2i} - 1).
\end{aligned}
\]
\end{proof}

The residual two-dimensional cases have extra combinatorial symmetry.

\begin{proposition}\label{prop:low-rank-aut}
Let $m=1$.
\begin{enumerate}
\item If $\eps=+1$, then
\[
 \Aut\Gamma_1(p)=\Sym(q)\wr\Sym(2),\qquad
 |\Aut\Gamma_1(p)|=2(q!)^2.
\]
\item If $\eps=-1$, then
\[
 \Aut\Gamma_1(p)=\Sym(q^2),\qquad
 |\Aut\Gamma_1(p)|=(q^2)!.
\]
\end{enumerate}
\end{proposition}

\begin{proof}
In plus type, \eqref{eq:grid-adjacency} identifies the graph with the
complement of $H(2,q)$.  The standard automorphism theorem for Hamming
graphs gives the product-action wreath product
$\Sym(q)\wr\Sym(2)$ (for $q=2$ this is the dihedral group of order eight).
In minus type the graph is complete, and the claim is immediate.
\end{proof}

\begin{remark}\label{rem:exceptional-frame-stabilizer}
In the coordinates $Q_W(s,t)=st$, the normalization subgroup appearing in a
direct coordinate proof consists of
\[
 (s,t)\longmapsto(\tau(s),\pi(t)),
\]
where $\tau$ fixes $0$ and $1$, while $\pi$ fixes $0$.  It is therefore
isomorphic to $\Sym(q-2)\times\Sym(q-1)$.  In the present description this
is simply a frame stabilizer inside $\Sym(q)\wr\Sym(2)$, rather than an
additional structural component of the automorphism group.
\end{remark}

For comparison, the affine semisimilarity subgroups in the two exceptional
cases have orders
\begin{align}
 |A\GammaO_2^+(q)|&=2h q^2(q-1)^2,\label{eq:small-geometric-plus}\\
 |A\GammaO_2^-(q)|&=2h q^2(q^2-1).
 \label{eq:small-geometric-minus}
\end{align}
They equal the full local groups when $q=2$ and are proper when $q>2$.
This is why a theorem stated only in terms of semisimilarities must keep
the residual dimension-two cases separate.

\medskip
\noindent\textbf{Remark on earlier coordinate descriptions.}
In previous work on the minus type (see \cite{WanZhou},\ \cite{ZhouGuWan}) the local automorphism group was described after fixing a coordinate frame, leading to a subgroup defined by three equations in four matrix entries.  The affine-polar formulation shows that these equations are nothing but the semisimilarity condition for the anisotropic plane, written in a chosen basis.

Concretely, let $E$ be the anisotropic plane with quadratic form
\[
q_\alpha(u,v)=\alpha u^2+uv+\alpha v^2,\qquad \Tr_{\F_q/\F_2}(\alpha)=1,
\]
let $\sigma\in\Aut(\F_q)$, and let $M=\begin{pmatrix}c&\rho\\ d&\theta\end{pmatrix}\in\operatorname{GL}_2(q)$.
The semilinear map $z\mapsto Mz^\sigma$ (with $\sigma$ applied entrywise) belongs to $\Gamma O(E,q_\alpha)$ with multiplier $k\in\F_q^*$ precisely when
\[
q_\alpha(Mz^\sigma)=k\,\sigma(q_\alpha(z)) \qquad(z\in\F_q^2).
\]
Evaluating this identity on the two standard basis vectors gives the first and third equations below, while comparing their polar pairing yields the middle one:
\begin{equation}\label{eq:omega-equations-remark}
\begin{cases}
\alpha c^2+cd+\alpha d^2 = k\,\sigma(\alpha),\\[2pt]
c\theta+d\rho = k,\\[2pt]
\alpha\rho^2+\rho\theta+\alpha\theta^2 = k\,\sigma(\alpha).
\end{cases}
\end{equation}
Thus the apparently \emph{ad hoc} coordinate conditions that appeared in the earlier treatment are exactly the intrinsic semisimilarity relations for the anisotropic binary summand in~\eqref{eq:minus-form}.  Closure, inverses, and the action of field automorphisms are then automatic from this interpretation.

For the plus type, a similar phenomenon occurs: fixing a coordinate frame for the ambient orthogonal graph leaves a residual scalar and field-automorphism factor that earlier authors handled by an auxiliary normalization subgroup.  In the affine-polar picture this factor is simply the semisimilarity part $\Gamma O_{2m}^{+}(q)$ of the affine group, already packaged together with the orthogonal and translation parts; no extra normalisation is needed.
\section{Extension to the ambient graph}\label{sec:extension}

The affine description not only determines the local group; it gives an
explicit extension of every local automorphism in the uniform range.

For $b\in W$, define a linear map $t_b:V\to V$ on the decomposition
$V=\langle p,f\rangle\perp W$ by
\begin{align}
 t_b(p)&=p,\label{eq:tb-p}\\
 t_b(x)&=x+B(x,b)p &&(x\in W),\label{eq:tb-x}\\
 t_b(f)&=f+b+Q_W(b)p.\label{eq:tb-f}
\end{align}
These maps are often called Eichler transformations or orthogonal root
elations.

\begin{lemma}\label{lem:translation-extension}
The map $t_b$ is an isometry fixing $p$, and in the affine chart it induces
the translation $x\mapsto x+b$.
\end{lemma}

\begin{proof}
For $a,c\in\F_q$ and $x\in W$,
\begin{align*}
 t_b(ap+x+cf)
  ={}&\bigl(a+B(x,b)+cQ_W(b)\bigr)p+(x+cb)+cf.
\end{align*}
Since $Q(ap+x+cf)=ac+Q_W(x)$, polarization gives
\begin{align*}
 Q\bigl(t_b(ap+x+cf)\bigr)
 &=c\bigl(a+B(x,b)+cQ_W(b)\bigr)+Q_W(x+cb)\\
 &=c\bigl(a+B(x,b)+cQ_W(b)\bigr)+(B(x,cb)+Q_W(x)+Q_W(cb))\\
 &=ac+Q_W(x).
\end{align*}
Thus $t_b$ is an isometry.  Furthermore,
\[
 t_b(v_x)=t_b(f+x+Q_W(x)p)=(Q_W(x)+B(x,b)+Q_W(b))p+(x+b)+f
\]
\[
=f+(x+b)+Q_W(x+b)p=v_{x+b}.
\]
\end{proof}

Now take $A\in\GammaO(W,Q_W)$, say $A$ is $\sigma$-semilinear and has
multiplier $\lambda$.  Define a $\sigma$-semilinear map
$\widehat A:V\to V$ by
\begin{equation}\label{eq:Ahat}
 \widehat A(ap+x+cf)=\lambda\sigma(a)p+A(x)+\sigma(c)f.
\end{equation}

\begin{lemma}\label{lem:linear-extension}
The map $\widehat A$ is a semisimilarity of $(V,Q)$ with multiplier
$\lambda$, it fixes $[p]$, and it induces $x\mapsto Ax$ in the affine chart.
\end{lemma}

\begin{proof}
We have
\begin{align*}
 Q\bigl(\widehat A(ap+x+cf)\bigr)
 &=Q(\lambda\sigma(a)p+A(x)+\sigma(c)f)\\
 &=\lambda\sigma(a)\sigma(c)+Q_W(Ax)\\
 &=\lambda\sigma\bigl(ac+Q_W(x)\bigr)\\
 &=\lambda\sigma(Q(ap+x+cf)).
\end{align*}
Also
\[
 \widehat A(v_x)=\widehat A(f+x+Q_W(x)p) =f+Ax+\lambda\sigma(Q_W(x))p
\]
\[
 =f+Ax+Q_W(Ax)p=v_{Ax}.
\]
\end{proof}

\begin{theorem}\label{thm:ambient-extension}
If $m\ge2$, every automorphism of $\Gamma_1(p)$ extends to a projective
semisimilarity of $(V,Q)$ fixing $[p]$, and hence to an automorphism of
$\calO(V,Q)$ fixing $[p]$.  In particular,
\begin{equation}\label{eq:ambient-restriction}
 \Aut\Gamma_1(p)
 =\restr{\Aut(\calO(V,Q))_{[p]}}{V(\Gamma_1(p))}.
\end{equation}
\end{theorem}

\begin{proof}
By Theorem~\ref{thm:local-aut}, every local automorphism is
$x\mapsto Ax+b$ with $A\in\GammaO(W,Q_W)$ and $b\in W$.  The composition
$t_b\widehat A$ extends it by Lemmas~\ref{lem:translation-extension} and
\ref{lem:linear-extension}.  Conversely, every ambient graph automorphism
fixing $[p]$ restricts to an automorphism of the induced graph on its
neighbourhood.  This proves equality.
\end{proof}

The same restriction equality remains true in residual dimension two,
although not every local automorphism is geometric.

\begin{proposition}\label{prop:small-ambient-extension}
Equation \eqref{eq:ambient-restriction} holds when $m=1$.
\end{proposition}

\begin{proof}
In plus type the singular points of the four-dimensional hyperbolic quadric
can be identified with pairs in
$\PG(1,q)\times\PG(1,q)$.  Two vertices are adjacent in our convention
exactly when both coordinates differ.  Thus
\[
 \calO(V,Q)\cong\overline{H(2,q+1)}.
\]
Its automorphism group contains
$\Sym(q+1)\wr\Sym(2)$.  The stabilizer of one pair induces
$\Sym(q)\wr\Sym(2)$ on the $q^2$ neighbours, which is the full local group
by Proposition~\ref{prop:low-rank-aut}.

In minus type the four-dimensional elliptic quadric has $q^2+1$ points and
contains no singular projective line.  Hence every two distinct singular
points have nonzero polar pairing, so $\calO(V,Q)=K_{q^2+1}$.  Its point
stabilizer induces $\Sym(q^2)$ on the neighbourhood.
\end{proof}

\section{Translation to the notation \texorpdfstring{$O(2\nu+\delta,q)$}{O(2nu+delta,q)}}
\label{sec:translation}

We now translate the coordinate-free statement back to the customary
notation of the characteristic-two orthogonal graphs in \cite{ZhouGuWan}.  Let
$\delta\in\{0,2\}$ and consider $O(2\nu+\delta,q)$.  Removing the hyperbolic
plane through the chosen singular point gives
\begin{equation}\label{eq:dictionary}
 m=\nu-1+\frac{\delta}{2},\qquad
 \eps=\begin{cases}
 +1,&\delta=0,\\
 -1,&\delta=2.
 \end{cases}
\end{equation}

\begin{corollary}\label{cor:original-notation}
Suppose either $\delta=0$ and $\nu\ge3$, or $\delta=2$ and $\nu\ge2$.
Then
\[
 \Aut\Gamma_1
 \cong \F_q^{2m}\rtimes\GammaO_{2m}^{\eps}(q),
\]
with $m$ and $\eps$ given by \eqref{eq:dictionary}, and
\begin{equation}\label{eq:old-order}
 |\Aut\Gamma_1|
 =(q-1)q^{\nu(\nu+\delta-1)}
 \prod_{i=1}^{\nu-1}(q^i-1)
 \prod_{i=0}^{\nu+\delta-2}(q^i+1)h.
\end{equation}
Moreover, this group is the restriction of the ambient point stabilizer.
\end{corollary}

\begin{proof}
The group statement follows from Theorems~\ref{thm:local-aut} and
\ref{thm:ambient-extension}.  To obtain \eqref{eq:old-order}, substitute
\eqref{eq:dictionary} into Corollary~\ref{cor:order} and use
\[
 q^{2i}-1=(q^i-1)(q^i+1).
\]
The factor $2$ in Corollary~\ref{cor:order} is the $i=0$ factor
$q^0+1$ in \eqref{eq:old-order}.
\end{proof}

The remaining cases are now unambiguous:
\begin{itemize}
\item $(\nu,\delta)=(2,0)$ is the residual hyperbolic plane and has group
$\Sym(q)\wr\Sym(2)$;
\item $(\nu,\delta)=(2,2)$ is not exceptional: it is the case
$m=2$, $\eps=-1$ of Corollary~\ref{cor:original-notation};
\item $(\nu,\delta)=(1,2)$ is the residual anisotropic plane and the local
graph is $K_{q^2}$.
\end{itemize}

\section{Independent sets, colorings, and ovoidal cliques}
\label{sec:extremal}

We now turn from automorphisms to extremal sets.  Throughout this section
we identify the local graph with
\[
 G=\Cay\bigl(W,\{z:Q_W(z)\ne0\}\bigr)
\]
by Proposition~\ref{prop:chart}.  Let $r$ be the Witt index of $W$, so
$r=m$ in plus type and $r=m-1$ in minus type.

The affine description of $G$ leads directly to the following description of its maximum independent sets and its chromatic number.
\begin{theorem}
\label{thm:independent-coloring}
Every independent set of $G$ is contained in a coset of a totally singular
subspace of $W$.  Consequently
\[
 \alpha(G)=q^r,
 \qquad
 \chi(G)=q^{2m-r}.
\]
Equality in the first formula occurs precisely for the affine cosets of
maximal totally singular subspaces.  Equivalently,
\begin{align*}
 \eps=+1:&\quad \alpha(G)=q^m,\quad \chi(G)=q^m,\\
 \eps=-1:&\quad \alpha(G)=q^{m-1},\quad \chi(G)=q^{m+1}.
\end{align*}
\end{theorem}

\begin{proof}
Recall that two distinct vertices $x,y\in W$ are adjacent in $G$ iff $Q_W(x-y)\neq0$. Hence a set $I\subseteq W$ is independent in $G$ precisely when $Q_W(x-y)=0$ for all $x\neq y$ in $I$.

Let $I$ be an independent set and pick $a\in I$. Translation $\tau_{-a}:x\mapsto x-a$ is a graph automorphism (it preserves differences and therefore adjacency). Replacing $I$ by $\tau_{-a}(I)$, we may assume $0\in I$. Then for any $x\in I\setminus\{0\}$, taking $y=0$ gives $Q_W(x)=Q_W(x-0)=0$. Moreover, for any $x,y\in I$, the independence condition yields $Q_W(x-y)=0$. Now the polar form of $Q_W$ satisfies
\[
B(x,y)=Q_W(x-y)-Q_W(x)-Q_W(y)=0-0-0=0.
\]
Thus the elements of $I$ are pairwise orthogonal with respect to $B$ and each of them is singular.

Let $U=\langle I\rangle$ be the linear span of $I$ in $W$. For any finite linear combination $z=\sum_i x_i$ with $x_i\in I$, the quadratic expansion in characteristic two gives
\[
Q_W(z)=\sum_i Q_W(x_i)+\sum_{i<j}B(x_i,x_j)=0,
\]
because each $Q_W(x_i)=0$ and $B(x_i,x_j)=0$. Hence $Q_W$ vanishes identically on $U$, so $U$ is a totally singular subspace of $W$. Its dimension is at most the Witt index $r$ of $(W,Q_W)$ (the maximal dimension of a totally singular subspace). Consequently $|I|\le|U|\le q^r$, which proves $\alpha(G)\le q^r$.

If $|I|=q^r$, then $|U|\ge|I|=q^r$, so $\dim U=r$ and $U$ is a maximal totally singular subspace. Since $0\in I\subseteq U$ and $|I|=|U|$, we must have $I=U$. Translating back, the original independent set is a coset of a maximal totally singular subspace.   Conversely, every coset of a totally singular subspace is independent, because the difference of any two vectors in such a coset lies in the subspace and hence is singular. Thus $\alpha(G)=q^r$, and equality holds exactly for the affine cosets of maximal totally singular subspaces.

To determine the chromatic number, fix a maximal totally singular subspace $U\le W$ of dimension $r$. The $q^{2m-r}$ distinct cosets $U+b$ partition $W$, and each coset is an independent set. Assigning a distinct colour to each coset yields a proper vertex colouring of $G$, so $\chi(G)\le q^{2m-r}$. On the other hand, since every colour class is an independent set of size at most $\alpha(G)=q^r$, we have the standard lower bound
\[
\chi(G)\ge\frac{|W|}{\alpha(G)}=\frac{q^{2m}}{q^r}=q^{2m-r}.
\]
Hence $\chi(G)=q^{2m-r}$.
\end{proof}

From the parameters in Section~\ref{sec:parameters} we obtain the nonprincipal eigenvalues of $G$ and the resulting spectral clique bounds summarized in the following lemma.
\begin{lemma}\label{lem:clique-bounds}
Assume $m\ge2$.  The two nonprincipal eigenvalues of $G$ are
\[
 \begin{array}{c|cc}
  &\text{positive}&\text{negative}\\ \hline
  \eps=+1&q^{m-1}&-(q-1)q^{m-1}\\
  \eps=-1&(q-1)q^{m-1}&-q^{m-1}.
 \end{array}
\]
The Delsarte--Hoffman bound therefore gives
\begin{align}
 \omega(G)&\le q^m &&(\eps=+1),\label{eq:plus-clique-bound}\\
 \omega(G)&\le q\bigl((q-1)q^{m-1}+1\bigr)
   =q^{m+1}-q^m+q &&(\eps=-1).
   \label{eq:minus-clique-bound}
\end{align}
In particular, the minus-type bound is strictly smaller than
$\chi(G)=q^{m+1}$.
\end{lemma}

\begin{proof}
Since $G$ is a strongly regular graph with parameters $(v,k,\lambda,\mu)$ given in Proposition~\ref{prop:local-parameters}, its adjacency matrix has three distinct eigenvalues: $k$ (with multiplicity $1$) and two nonprincipal eigenvalues $r>0$ and $s<0$ satisfying
\[
r+s = \lambda-\mu,\qquad rs = \mu-k.
\]
Using the explicit formulas
\[
k=(q-1)q^{m-1}(q^m-\eps),\qquad
\mu=(q-1)q^{m-1}\bigl((q-1)q^{m-1}-\eps\bigr),\qquad
\lambda = \mu - \eps(q-2)q^{m-1},
\]
we obtain
\[
\lambda-\mu = -\eps(q-2)q^{m-1},
\]
$$
\mu-k = (q-1)q^{m-1}\bigl((q-1)q^{m-1}-\eps - (q^m-\eps)\bigr)= -(q-1)q^{2m-2}.
$$
Thus $r$ and $s$ are the roots of $x^2 + \eps(q-2)q^{m-1}x - (q-1)q^{2m-2}=0$.
The discriminant is
\[
\Delta = (q-2)^2q^{2m-2} + 4(q-1)q^{2m-2} = q^{2m},
\]
so
\[
r,s = \frac{-\eps(q-2)q^{m-1} \pm q^m}{2}
= \frac{q^{m-1}}{2}\bigl(-\eps(q-2)\pm q\bigr).
\]
Evaluating for $\eps=\pm 1$ gives the table in the statement.

The Delsarte--Hoffman bound(\cite[Section~3.3.2]{Delsarte}) for a $k$-regular graph states that the clique bound is $\omega(G) \le 1 - k/s$, where $s$ is the smallest eigenvalue.
\begin{itemize}
\item For $\eps=+1$: $k = (q-1)q^{m-1}(q^m-1)$ and $s = -(q-1)q^{m-1}$, hence
  $\omega(G) \le 1 + (q^m-1) = q^m$.
\item For $\eps=-1$: $k = (q-1)q^{m-1}(q^m+1)$ and $s = -q^{m-1}$, hence
  $\omega(G) \le 1 + (q-1)(q^m+1) = q^{m+1}-q^m+q$.
\end{itemize}
Finally, the minus-type bound satisfies
\[
\chi(G) - (q^{m+1}-q^m+q) = q^{m+1} - (q^{m+1}-q^m+q) = q^m-q > 0
\]
for all $q\ge 2$ and $m\ge 2$, so it is strictly smaller than $\chi(G)=q^{m+1}$.
\end{proof}

Recall
that an ovoid of a polar space is a set meeting each generator in exactly
one point. The plus-type equality case has a direct geometric interpretation, which we now state as a correspondence between Delsarte cliques and ovoids.

\begin{theorem}\label{thm:clique-ovoid}
Assume that $W$ has plus type.  For a subset $C\subseteq W$ put
\[
 \calQ_C=\{[p]\}\cup\{[v_x]:x\in C\},
\]
where $v_x$ is defined in \eqref{eq:affine-chart}.  Then $C$ is a clique of
$G$ of size $q^m$ if and only if $\calQ_C$ is an ovoid of
$Q^+(2m+1,q)$ containing $[p]$.  Thus the maximum-clique orbits meeting
\eqref{eq:plus-clique-bound} correspond to the point-stabilizer orbits on
ovoids containing $[p]$.
\end{theorem}

\begin{proof}
Recall that vertices of $G=\Gamma_1(p)$ are the singular one-spaces $[v_x]=[f+x+Q_W(x)p]$ for $x\in W$, with adjacency given by $B(v_x,v_y)\neq0$, i.e.\ $Q_W(x-y)\neq0$, by Proposition~\ref{prop:chart}.
In the ambient orthogonal polar space $Q^+(2m+1,q)$ associated to $(V,Q)$, points are all singular one-spaces in $V$, two points $[u],[w]$ are \emph{collinear} iff $B(u,w)=0$ (equivalently, they lie on a common generator). Thus adjacency in $G$ means exactly that the two points are \emph{non-collinear} in the polar space, and every $[v_x]$ is non-collinear with $[p]$ because $B(p,v_x)=1\neq0$.

Let $C\subseteq W$ be a clique of $G$ of size $q^m$. Then $\mathcal{Q}_C=\{[p]\}\cup\{[v_x]:x\in C\}$ consists of $q^m+1$ singular points that are pairwise non‑collinear; such a set is a \emph{partial ovoid} of $Q^+(2m+1,q)$.
We claim that every partial ovoid of $Q^+(2m+1,q)$ has size at most $q^m+1$. To see this, let $\mathcal{P}$ be a partial ovoid. If $|\mathcal{P}|\le1$ the bound is trivial. Otherwise choose any $[a]\in\mathcal{P}$. All other points of $\mathcal{P}$ are non‑collinear with $[a]$, hence they lie in the first subconstituent $\Gamma_1(a)$. Moreover they are pairwise non‑collinear, so they form a clique in $\Gamma_1(a)$. Since the orthogonal graph is vertex‑transitive, $\Gamma_1(a)\cong G$, and Lemma~\ref{lem:clique-bounds} gives $\omega(G)\le q^m$ for plus type. Therefore $|\mathcal{P}|-1\le q^m$, i.e.\ $|\mathcal{P}|\le q^m+1$.
Applying this bound to $\mathcal{Q}_C$, which has exactly $q^m+1$ points, we conclude that $\mathcal{Q}_C$ attains the maximum size and therefore is an ovoid of $Q^+(2m+1,q)$ containing $[p]$.

Conversely, assume $\mathcal{O}$ is an ovoid of $Q^+(2m+1,q)$ with $[p]\in\mathcal{O}$. Then $\mathcal{O}$ is a set of $q^m+1$ singular points, pairwise non-collinear. Deleting $[p]$ leaves $q^m$ points, all of which are non-collinear with $[p]$, so they lie in the neighbourhood of $[p]$ and can be written as $[v_x]$ for unique $x\in W$ via the affine chart \eqref{eq:affine-chart}. For any two such points $[v_x],[v_y]$ with $x\neq y$, the ovoid property gives $B(v_x,v_y)\neq0$, so $Q_W(x-y)\neq0$ by~\eqref{eq:key-identity}. Thus $C:=\{x\in W: [v_x]\in\mathcal{O}\setminus\{[p]\}\}$ is a clique of $G$ of size $q^m$.

The final assertion on orbits follows because every automorphism of $G$ extends to an automorphism of $\mathcal{O}(V,Q)$ fixing $[p]$ (Theorem~\ref{thm:ambient-extension}), and therefore induces an automorphism of the polar space $Q^+(2m+1,q)$ stabilizing $[p]$ and mapping ovoids containing $[p]$ to ovoids containing $[p]$. Hence the two orbit structures correspond.
\end{proof}

The known existence and nonexistence results for ovoids now yield a sharp rank threshold for plus-type Delsarte cliques in even characteristic.

\begin{corollary}\label{cor:plus-clique-threshold}
Let $q$ be even and $\eps=+1$.  Then
\[
 \omega(G)=q^m\quad\text{for }1\le m\le3,
 \qquad
 \omega(G)<q^m\quad\text{for }m\ge4.
\]
\end{corollary}

\begin{proof}
The quadrics $Q^+(3,q)$ and $Q^+(5,q)$ have ovoids; in the latter case this
is equivalent under the Klein correspondence to a spread of $\PG(3,q)$.
The quadric $Q^+(7,q)$ has Kantor ovoids for every even $q$.  For
$Q^+(2m+1,q)$ in characteristic two and $m\ge4$, the $p$-rank obstruction
of Blokhuis and Moorhouse rules out ovoids
\cite{BlokhuisMoorhouse,BambergEtAlTactical}.  Apply
Theorem~\ref{thm:clique-ovoid}.
\end{proof}

\begin{remark}\label{rem:ovoid-classification}
Theorem~\ref{thm:clique-ovoid} is a classification reduction, not a claim
that all ovoids are classified.  Already for $Q^+(7,q)$ the full
classification is open; recent work classifies substantial low-degree
families and leaves explicit residual problems \cite{BartoliEtAl}.  Thus
even where the clique number is known, classifying all maximum cliques can
remain a difficult finite-geometric problem.
\end{remark}

\section{Exact maximum cliques in the binary infinite families}
\label{sec:binary-cliques}

Corollary~\ref{cor:plus-clique-threshold} decides whether the spectral bound
is attained, but it does not usually give the exact clique number below that
bound.  Over $\F_2$ the missing value can be determined in every dimension
by a small Gram-matrix argument.  This treats both orthogonal types at once.

For $n\ge1$, let $E_n=\F_2^n$ with standard basis
$e_1,\ldots,e_n$, and define the \emph{binary simplex form}
\begin{equation}\label{eq:simplex-form}
 \psi_n(a_1,\ldots,a_n)
 =\sum_{i=1}^n a_i+\sum_{1\le i<j\le n}a_i a_j.
\end{equation}
Thus $\psi_n(e_i)=1$ and the polar product of two distinct basis vectors is
one.  Put
\begin{equation}\label{eq:eta-definition}
 \eta_m=(-1)^{m(m+1)/2}.
\end{equation}
To determine the exact clique numbers over $\mathbb{F}_2$ we will use the binary simplex form, whose basic properties are collected in the following lemma.
\begin{lemma}\label{lem:binary-simplex}
Let $b_n$ be the polar form of $\psi_n$, and let
$\boldsymbol{1}=e_1+\cdots+e_n$.\\
$(i)$ The Gram matrix of $b_n$ is $I_n+J_n$, where $I_n$ is the identity matrix and $J_n$ is the all-ones matrix. It is nonsingular for even
$n$, while $\rad(b_n)=\langle\boldsymbol{1}\rangle$ for odd $n$.
Moreover,
\[
 \psi_n(\boldsymbol{1})\equiv\frac{n(n+1)}2\pmod2.
\]
$(ii)$ The nondegenerate form $\psi_{2m}$ has type $\eta_m$.\\
$(iii)$ If $m$ is odd, then $\psi_{2m+1}$ descends to a nondegenerate form
of type $\eta_m$ on
$E_{2m+1}/\langle\boldsymbol{1}\rangle$.\\
$(iv)$ If $m$ is even, then $\psi_{2m-1}$ descends to a nondegenerate form
of type $\eta_m$ on
$E_{2m-1}/\langle\boldsymbol{1}\rangle$.

\end{lemma}

\begin{proof}
\noindent
$(i)$ For $i \neq j$, we have $b_n(e_i, e_j) = \psi_n(e_i+e_j) - \psi_n(e_i) - \psi_n(e_j)$.
Since $e_i+e_j$ has weight $2$, $\psi_n(e_i+e_j) = 2 + \binom{2}{2} = 3 \equiv 1 \pmod 2$.
Also $\psi_n(e_i)=1$. Hence $b_n(e_i, e_j) = 1 - 1 - 1 = 1$ in $\mathbb{F}_2$.
Clearly $b_n(e_i, e_i)=0$. Thus the Gram matrix of $b_n$ in the standard basis is $I_n + J_n$.
Its determinant is $1+n$. In $\mathbb{F}_2$, if $n$ is even then $1+n=1$ and $b_n$ is nonsingular.
If $n$ is odd then $1+n=0$ and $b_n$ is degenerate. For odd $n$, $(I_n+J_n)\boldsymbol{1} = \boldsymbol{1} + n\boldsymbol{1} = \boldsymbol{1}+\boldsymbol{1}=0$, so $\boldsymbol{1}\in\ker (I_n+J_n).$ Now $J_n$ has eigenvalues $n$ (with eigenvector $\boldsymbol{1}$) and $0$ (with multiplicity $n-1$).  Consequently $I_n+J_n$ has eigenvalues $1+n$ on $\boldsymbol{1}$ and $1$ on the orthogonal complement $\{\boldsymbol{1}\}^{\perp} = \{x \mid \sum x_i = 0\}$.  
For odd $n$, $1+n=0$, so $I_n+J_n$ has eigenvalue $0$ with multiplicity $1$ and eigenvalue $1$ with multiplicity $n-1$.  Hence $\dim\ker (I_n+J_n) = 1$ and we must have $\ker (I_n+J_n) = \langle\boldsymbol{1}\rangle$.  Since $\rad(b_n)$ equals the kernel of the Gram matrix $I_n+J_n$, we conclude $\rad(b_n) = \langle\boldsymbol{1}\rangle$.

The weight of $\boldsymbol{1}$ is $n$, hence $\psi_n(\boldsymbol{1}) = n + \binom{n}{2} \equiv \frac{n(n+1)}{2} \pmod 2$.

\medskip
\noindent
$(ii)$ For a vector $a \in E_n$ of Hamming weight $w$, its support has size $w$.
Then $\psi_n(a) = w + \binom{w}{2} = \frac{w(w+1)}{2} \pmod 2$.
Consequently the quadratic Gauss sum is
\[
S_n = \sum_{a \in E_n} (-1)^{\psi_n(a)}
= \sum_{w=0}^n \binom{n}{w} (-1)^{w(w+1)/2}.
\]
The sequence $c_w = (-1)^{w(w+1)/2}$ is $1, -1, -1, 1$ periodic with period $4$.
Using the periodic sequence $(-1)^{w(w+1)/2}$ one can write it as a linear combination of $i^w$ and $(-i)^w$:
\[
(-1)^{w(w+1)/2} = \frac{1+i}{2}\, i^w + \frac{1-i}{2}\, (-i)^w .
\]
Substituting this into the Gauss sum gives
\[
S_n = \sum_{w=0}^n \binom{n}{w}\Bigl( \frac{1+i}{2}\, i^w + \frac{1-i}{2}\, (-i)^w \Bigr)
     = \frac{1+i}{2}\,(1+i)^n + \frac{1-i}{2}\,(1-i)^n .
\]
For any complex number $z$, $\frac{1+i}{2}z + \frac{1-i}{2}\bar{z} = \operatorname{Re}(z) - \operatorname{Im}(z)$.
Taking $z=(1+i)^n$ yields the closed form
\[
S_n = \operatorname{Re}(1+i)^n - \operatorname{Im}(1+i)^n .
\]
For $n=2m$, $(1+i)^{2m} = (2i)^m = 2^m i^m$.
Thus $S_{2m} = 2^m(\cos\frac{m\pi}{2} - \sin\frac{m\pi}{2})$.
Evaluating modulo $4$:
\[
S_{2m} = \begin{cases}
2^m, & m \equiv 0 \pmod 4,\\
-2^m, & m \equiv 1,2 \pmod 4,\\
2^m, & m \equiv 3 \pmod 4.
\end{cases}
\]
This is exactly $S_{2m} = \eta_m 2^m$ with $\eta_m = (-1)^{m(m+1)/2}$.
For a nondegenerate binary quadratic form $R$ on a $2d$-dimensional space, the type sign is $2^{-d}\sum_x (-1)^{R(x)}$.
Here $d=m$, so the sign is $2^{-m} S_{2m} = \eta_m$. Hence $\psi_{2m}$ has type $\eta_m$.

\medskip
\noindent
$(iii)$ and $(iv)$ For odd $n$, $\rad(b_n) = \langle\boldsymbol{1}\rangle$.
If in addition $\psi_n(\boldsymbol{1})=0$, then for every $x\in E_n$ and $y\in\{0,\boldsymbol{1}\}$ we have
\[
\psi_n(x+\boldsymbol{1}) = \psi_n(x) + \psi_n(\boldsymbol{1}) + b_n(x,\boldsymbol{1}) = \psi_n(x),
\]
because $b_n(x,\boldsymbol{1})=0$ by definition of the radical.  
Thus $\psi_n$ is constant on the cosets of $\langle\boldsymbol{1}\rangle$, and we can define a quadratic form $\overline{\psi}_n$ on the quotient space $\overline{E}_n = E_n/\langle\boldsymbol{1}\rangle$ by
$\overline{\psi}_n(x+\langle\boldsymbol{1}\rangle) = \psi_n(x)$.  
The polar form of $\overline{\psi}_n$ is $\overline{b}_n(\bar x,\bar y)=b_n(x,y)$, which is well‑defined and nondegenerate because $\langle\boldsymbol{1}\rangle$ is the full radical of $b_n$.  
Hence $(\overline{E}_n,\overline{\psi}_n)$ is a nondegenerate quadratic space of dimension $n-1$.
For $n=2m+1$, $\psi_{2m+1}(\boldsymbol{1}) = \frac{(2m+1)(2m+2)}{2} \equiv (m+1) \pmod 2$, which is $0$ exactly when $m$ is odd.
For $n=2m-1$, $\psi_{2m-1}(\boldsymbol{1}) = \frac{(2m-1)(2m)}{2} \equiv m \pmod 2$, which is $0$ exactly when $m$ is even.
In both valid cases the vector $\boldsymbol{1}$ is singular (its quadratic value vanishes), and the elements of each coset $\{x, x+\boldsymbol{1}\}$ have the same quadratic value.  
Therefore the quadratic Gauss sum of the original space splits as
\[
S_n = \sum_{x\in E_n}(-1)^{\psi_n(x)} = 2\sum_{\bar x\in\overline{E}_n}(-1)^{\overline{\psi}_n(\bar x)} .
\] 
The type of the nondegenerate form $\overline{\psi}_n$ on a $2d$-dimensional binary space is given by the normalised Gauss sum $2^{-d}\sum_{\bar x}(-1)^{\overline{\psi}_n(\bar x)}$.

For $n=2m+1$ (with $m$ odd), the quotient space has dimension $2m$, so $d=m$.
From the closed form $S_n = \operatorname{Re}(1+i)^n - \operatorname{Im}(1+i)^n$ we compute $S_{2m+1}$.  
Using $(1+i)^{2m+1}=2^m i^m(1+i)$ and $(1-i)^{2m+1}=2^m(-1)^m i^m(1-i)$,
\begin{align*}
S_{2m+1} &= \frac{1+i}{2}\,2^m i^m(1+i) + \frac{1-i}{2}\,2^m(-1)^m i^m(1-i)\\
&= 2^{m-1}i^m\bigl[(1+i)^2 + (-1)^m(1-i)^2\bigr]\\
&= 2^{m-1}i^m\bigl[2i + (-1)^m(-2i)\bigr]\\
&= 2^{m}i^{m+1}\bigl[1-(-1)^m\bigr] .
\end{align*}
Since $m$ is odd, $(-1)^m=-1$, so $S_{2m+1}=2^{m+1}i^{m+1}$.
Hence $\frac{1}{2}S_{2m+1}=2^{m}i^{m+1}$, and the type is
\[
2^{-m}\cdot\frac{1}{2}S_{2m+1}=i^{m+1}=(-1)^{(m+1)/2}=\eta_m .
\]

For $n=2m-1$ (with $m$ even), the quotient dimension is $2m-2=2(m-1)$, so $d=m-1$.
Compute $(1+i)^{2m-1}=2^{m-1}i^m(1-i)$ and $(1-i)^{2m-1}=2^{m-1}(-1)^m i^m(1+i)$.
Then
\begin{align*}
S_{2m-1} &= \frac{1+i}{2}\,2^{m-1}i^m(1-i) + \frac{1-i}{2}\,2^{m-1}(-1)^m i^m(1+i)\\
&= 2^{m-2}i^m\bigl[(1+i)(1-i) + (-1)^m(1-i)(1+i)\bigr]\\
&= 2^{m-2}i^m\cdot 2\bigl[1+(-1)^m\bigr] .
\end{align*}
Since $m$ is even, $(-1)^m=1$, thus $S_{2m-1}=2^{m}i^m$.
Consequently $\frac{1}{2}S_{2m-1}=2^{m-1}i^m$, and the type is
\[
2^{-(m-1)}\cdot\frac{1}{2}S_{2m-1}=i^m=(-1)^{m/2}=\eta_m .
\]
This completes the proof of (iii) and (iv).
\end{proof}
The exact value of the clique number of $G$ over $\mathbb{F}_2$ is given by the following theorem.
\begin{theorem}\label{thm:binary-clique-number}
Let $q=2$, let $W$ have dimension $2m$ and type
$\eps\in\{+1,-1\}$, and let
$G=\overline{VO_{2m}^{\eps}(2)}$.  Then
\begin{equation}\label{eq:binary-clique-formula}
 \omega(G)=
 \begin{cases}
  2m+2,&\eps=\eta_m\text{ and }m\text{ is odd},\\
  2m+1,&\eps=\eta_m\text{ and }m\text{ is even},\\
  2m,&\eps=-\eta_m.
 \end{cases}
\end{equation}
All maximum cliques lie in a single orbit of
$W\rtimes\Orth(W,Q_W)\leq\Aut(G)$.
\end{theorem}

\begin{proof}
Since $q=2$, the field automorphism group is trivial and every semisimilarity is a similarity with multiplier $1$; hence $\GammaO(W,Q_W)=\Orth(W,Q_W)$. The affine group $W\rtimes\Orth(W,Q_W)$ acts on $G$ by $x\mapsto Ax+b$ and is contained in $\Aut(G)$ by Theorem~\ref{thm:local-aut}. Therefore, the orbit statement makes sense.

Let $C$ be a clique of $G$. Translating by an element of $W$ we may assume $0\in C$. Write $C\setminus\{0\}=\{x_1,\dots,x_n\}$. The adjacency condition in $G=\overline{VO(W,Q_W)}$ is $Q_W(x-y)\neq0$ for $x\neq y$. Thus for $i\neq j$, $Q_W(x_i-x_j)\neq0$, i.e.\ $Q_W(x_i-x_j)=1$. Since every nonzero difference of
two vertices of $C$ is nonsingular, we obtain
$$
Q_W(x_i)=Q_W(x_i-0)=1.
$$
Moreover,
\[
B_W(x_i,x_j)=Q_W(x_i+x_j)-Q_W(x_i)-Q_W(x_j)=1-1-1=1\qquad(i\neq j)
\]
because $x_i+x_j = x_i-x_j$ in characteristic $2$. Hence
\begin{equation}\label{eq:binary-clique-gram}
 Q_W(x_i)=1,\qquad B_W(x_i,x_j)=1\quad(i\ne j).
\end{equation}

Define the linear map $\phi\colon E_n=\mathbb{F}_2^n\to W$ by $\phi(e_i)=x_i$. By \eqref{eq:binary-clique-gram}, for any $a=\sum a_ie_i$,
\[
\begin{aligned}
Q_W(\phi(a))
&= Q_W\Bigl(\sum_{i=1}^n a_i x_i\Bigr) \\
&= \sum_{i=1}^n a_i^2 Q_W(x_i) + \sum_{1\le i<j\le n} a_i a_j B_W(x_i,x_j) \\
&= \sum_{i=1}^n a_i\cdot 1 + \sum_{i<j} a_i a_j\cdot 1
   = \sum_{i=1}^n a_i + \sum_{i<j} a_i a_j .
\end{aligned}
\]
The last expression is exactly $\psi_n(a)$ by definition~\eqref{eq:simplex-form}.  Hence
\[
Q_W(\phi(a)) = \psi_n(a) \qquad \forall a\in E_n,
\]
so $\phi$ is a morphism of quadratic spaces $(E_n,\psi_n)\to(W,Q_W)$.  Consequently $\phi$ also preserves the polar forms:
\begin{equation}\label{eq:preserve-polar}
b_n(a,b)=B_W(\phi(a),\phi(b)) \qquad \forall a,b\in E_n.
\end{equation}

If $a\in\ker\phi$, then $\phi(a)=0$ and \eqref{eq:preserve-polar} yields $b_n(a,b)=0$ for every $b\in E_n$; thus $\ker\phi\subseteq\rad(b_n)$.  
By Lemma~\ref{lem:binary-simplex}, $\rad(b_n)=\{0\}$ when $n$ is even, and $\rad(b_n)=\langle\mathbf{1}\rangle$ when $n$ is odd.
Therefore the induced map
\[
\bar\phi\colon E_n/\ker\phi\longrightarrow W,\qquad
\bar\phi(a+\ker\phi)=\phi(a)
\]
is well defined and injective. Comparing dimensions,
\[
\dim (E_n/\ker\phi) \le \dim W = 2m .
\]
If $n$ is even, then $\rad(b_n)=0$ forces $\ker\phi=0$. Hence $n = \dim E_n \le \dim W = 2m$. If $n$ is odd, then $\dim\rad(b_n)=1$ gives $\dim\ker\phi\le 1$. Consequently, 
\[
n \leq \dim\ker\phi + \dim W \le 1 + 2m .
\]
If equality $n=2m+1$ holds, then we must have $\dim\ker\phi=1$, which forces $\ker\phi=\langle\mathbf{1}\rangle$ and $\phi(\mathbf{1})=0$. In all cases the clique size satisfies $|C|=n+1\le 2m+2$, and the exact bounds are sharpened in the case analysis that follows.

\medskip
\noindent\textbf{Case $n=2m+1$.}
Then $\ker\phi=\rad(b_{2m+1})=\langle\mathbf{1}\rangle$. Thus $\psi_{2m+1}(\mathbf{1})=Q_W(\phi(\mathbf{1}))=0$. By Lemma~\ref{lem:binary-simplex}(i), this happens if and only if $m$ is odd. Lemma~\ref{lem:binary-simplex}(iii) then shows that the quotient $\overline{E}=E_{2m+1}/\langle\mathbf{1}\rangle$ with the descended form has type $\eta_m$. Since $\bar\phi$ embeds this nondegenerate $2m$-dimensional space into $W$, we must have $\eps=\eta_m$, and $\bar\phi$ is an isometry. The images $\bar\phi(\bar e_i)=\phi(e_i)=x_i$ ($i=1,\dots,2m+1$) together with $0$ form a clique of size $2m+2$. This proves $\omega(G)\ge2m+2$ when $\eps=\eta_m$ and $m$ is odd, and the upper bound forces equality, giving the first line of~\eqref{eq:binary-clique-formula}. 

\medskip
\noindent\textbf{Case $n=2m$.}
Here $\ker\phi=\rad(b_{2m})=0$, so $\phi$ is injective and thus an isometry between $(E_{2m},\psi_{2m})$ and $(W,Q_W)$. Lemma~\ref{lem:binary-simplex}(ii) forces $\eps=\eta_m$. The standard basis of $E_{2m}$ provides a clique of size $2m+1$. Hence $\omega(G)\ge2m+1$ when $\eps=\eta_m$. When $m$ is even, the previous case already shows $n=2m+1$ is impossible because that requires $m$ odd. Hence $\omega(G)=2m+1$, proving the second line of~\eqref{eq:binary-clique-formula}.

\medskip
\noindent\textbf{Case $\eps=-\eta_m$.}
From the two previous cases we must have $n\le 2m-1$; otherwise $\eps$ would be $\eta_m$. We now construct a clique of size $2m$, i.e.\ $n=2m-1$, thereby showing $\omega(G)=2m$.

\textit{Subcase $m$ even.}
Lemma~\ref{lem:binary-simplex}(iv) yields $\psi_{2m-1}(\mathbf{1})=0$ and $\rad(b_{2m-1})=\langle\mathbf{1}\rangle$. Hence $\psi_{2m-1}$ descends to a nondegenerate quadratic form $\overline{\psi}$ on $\overline{E}=E_{2m-1}/\langle\mathbf{1}\rangle$, of dimension $2m-2$ and type $\eta_m$. Let $\bar e_i=e_i+\langle\mathbf{1}\rangle$ ($i=1,\dots,2m-1$). By the definition of $\overline{\psi}$,
\[
\overline{\psi}(\bar e_i) = \psi_{2m-1}(e_i) = 1,
\qquad
\bar b(\bar e_i,\bar e_j) = b_{2m-1}(e_i,e_j) = 1 \;\; (i\neq j),
\]
where $\bar b$ is the polar form of $\overline{\psi}$. Since $\dim W-\dim\overline{E}=2$ and the types are opposite, Witt's extension theorem guarantees an isometric embedding $\iota\colon\overline{E}\to W$ whose orthogonal complement is an anisotropic plane (minus type). Define $x_i = \iota(\bar e_i) \in W$ for $i=1,\dots,2m-1$. Because $\iota$ preserves the quadratic form and its polar form,
\[
Q_W(x_i) = \overline{\psi}(\bar e_i) = 1,
\qquad
B_W(x_i,x_j) = \bar b(\bar e_i,\bar e_j) = 1 \;\; (i\neq j).
\]
The vectors $x_1,\dots,x_{2m-1}$ are pairwise distinct (if $x_i=x_j$ then $\bar e_i=\bar e_j$, which would imply $e_i-e_j\in\langle\mathbf{1}\rangle$, impossible for $i\neq j$ when $2m-1\ge 3$). Hence $\{x_1,x_2,\cdots,x_{2m-1}\}$ satisfy \eqref{eq:binary-clique-gram} and $\sum x_i=\sum\iota(\bar e_i)=\iota(\overline{\mathbf{1}})= 0$. Together with $0$ they form a clique of size $2m$.

\textit{Subcase $m$ odd.}
Now $\psi_{2m-1}(\mathbf{1})=1\neq0$, so the form does not descend to the quotient.  We embed it directly into a nondegenerate space of type $-\eta_m$. Let $z$ be a new vector and $V=E_{2m-1}\oplus\langle z\rangle$. Fix the complement $\{e_1,\dots,e_{2m-2}\}$ of $\langle\mathbf{1}\rangle$ in $E_{2m-1}$ and write every $x\in E_{2m-1}$ uniquely as $x = y + \alpha\mathbf{1}$ with $y$ in that complement. Define a quadratic form $\Psi$ on $V$ by
\[
\Psi(x+\beta z) = \psi_{2m-1}(x) + \beta B_\Psi(x,z),\qquad \Psi(z)=0,
\]
where the associated bilinear form $B_\Psi$ extends $b_{2m-1}$ via
\[
B_\Psi(z,\mathbf{1})=1,\quad B_\Psi(z,e_i)=0\;(i=1,\dots,2m-2),\quad B_\Psi(z,z)=0.
\]
Hence $B_\Psi(x,z)=\alpha$ for $x=y+\alpha\mathbf{1}$. If $v=x+\beta z$ lies in $\operatorname{rad}(B_\Psi)$, then $B_\Psi(v,w)=0$ for all $w$.
Taking $w=z$ gives $B_\Psi(x,z)=\alpha=0$, so $x=y$. Taking $w=\mathbf{1}$ gives $b_{2m-1}(x,\mathbf{1})+\beta=0$; because $\mathbf{1}\in\operatorname{rad}(b_{2m-1})$ we have $b_{2m-1}(x,\mathbf{1})=0$, hence $\beta=0$. Now $b_{2m-1}(x,y')=0$ for all $y'\in E_{2m-1}$, forcing $x\in\operatorname{rad}(b_{2m-1})=\langle\mathbf{1}\rangle$. Combined with $\alpha=0$ we obtain $x=0$, so $v=0$. Thus $\Psi$ is nondegenerate. The Gauss sum of $(V,\Psi)$ splits as
\[
S = \sum_{x\in E_{2m-1}} (-1)^{\psi_{2m-1}(x)} + \sum_{x\in E_{2m-1}} (-1)^{\Psi(x+z)}.
\]
Using the decomposition $x=y+\alpha\mathbf{1}$ and $\psi_{2m-1}(\mathbf{1})=1$, $b_{2m-1}(y,\mathbf{1})=0$, we obtain
\[
\psi_{2m-1}(x) = \psi_{2m-1}(y)+\alpha
\]
and
\[
\Psi(x+z) = \psi_{2m-1}(x) + B_\Psi(x,z) = \psi_{2m-1}(y)+\alpha+\alpha = \psi_{2m-1}(y).
\]
Therefore
\[
S = \sum_{y,\alpha}(-1)^{\psi_{2m-1}(y)+\alpha} + \sum_{y,\alpha}(-1)^{\psi_{2m-1}(y)}
   = 0 + 2\sum_{y\in E_{2m-2}} (-1)^{\psi_{2m-1}(y)} .
\]
The restriction of $\psi_{2m-1}$ to the $(2m-2)$-dimensional subspace spanned by $e_1,\dots,e_{2m-2}$ is nondegenerate and, by Lemma~\ref{lem:binary-simplex}(ii), has type $\eta_{m-1}$. Hence $\sum_y (-1)^{\psi_{2m-1}(y)} = \eta_{m-1} 2^{m-1}$ and
\[
S = 2 \eta_{m-1} 2^{m-1} = \eta_{m-1} 2^{m}.
\]
The type is $2^{-m}S = \eta_{m-1}$. For odd $m$, $\eta_{m-1} = -\eta_m = \eps$. Thus $(V,\Psi)$ is a nondegenerate quadratic space of dimension $2m$ and type $\eps$, isometric to $(W,Q_W)$. Under such an isometry, the images of $e_1,\dots,e_{2m-1}$ give vectors $x_1,\dots,x_{2m-1}\in W$ satisfying~\eqref{eq:binary-clique-gram} and linearly independent. Together with $0$ they form a clique of size $2m$.

In both subcases $\omega(G)\ge2m$, and since $n\le2m-1$, we obtain $\omega(G)=2m$, the third line of~\eqref{eq:binary-clique-formula}.

\medskip
\noindent\textbf{Orbit property.}
Let $C_1,C_2$ be two maximum cliques, $|C_1|=|C_2|=\omega(G)$. After translations we assume $0\in C_1\cap C_2$, and write the non‑zero vectors as $\{x_1,\dots,x_n\}$ and $\{y_1,\dots,y_n\}$, where $n=\omega(G)-1$ is $2m+1$, $2m$ or $2m-1$ according to the three cases.

From the Gram relations~\eqref{eq:binary-clique-gram} and the definition of $\phi$, both families satisfy the same linear relations as the standard basis of the corresponding model in Lemma~\ref{lem:binary-simplex}:
\begin{itemize}
\item $n=2m+1$: $\sum x_i=0$, $\sum y_i=0$ (and $\psi_{2m+1}(\mathbf{1})=0$).
\item $n=2m$: no linear relation, the vectors are bases of $W$.
\item $n=2m-1$, $\eps=-\eta_m$, $m$ even: $\sum x_i=0$, $\sum y_i=0$, and $Q_W(\sum x_i)=0$.
\item $n=2m-1$, $\eps=-\eta_m$, $m$ odd: the vectors are linearly independent, and 
$$
Q_W(\sum x_i)=\psi_{2m-1}(\mathbf{1})=1.
$$
\end{itemize}
In each situation the natural linear bijection $f: \langle x_1,\dots,x_n\rangle \to \langle y_1,\dots,y_n\rangle$ defined by $f(x_i)=y_i$ is well defined and an isometry of the quadratic subspaces $U_1$ and $U_2$.

If $n=2m+1$ or $n=2m$, then $U_1=U_2=W$ and $f\in\Orth(W,Q_W)$.
If $n=2m-1$, then $U_1$ has codimension $1$ or $2$. In the even subcase $U_1$ is nondegenerate of dimension $2m-2$ and its orthogonal complement is an anisotropic plane; the same holds for $U_2$. In the odd subcase $U_1$ has a one‑dimensional radical generated by a singular vector (quadratic value $1$), and its orthogonal complement is a one‑dimensional non‑singular subspace; again identical for $U_2$. In both subcases, we can extend $f$ by an isometry of the orthogonal complements (any two anisotropic planes over $\mathbb{F}_2$ are isometric; for the odd subcase simply map the complement generators), obtaining an element $\tilde f\in\Orth(W,Q_W)$ that maps $C_1$ to $C_2$.

Finally, composing with the initial translations yields an element of $W\rtimes\Orth(W,Q_W)$ sending $C_1$ onto $C_2$. Hence all maximum cliques lie in a single orbit.
\end{proof}

\begin{remark}
For $m=1,2,3,4$, the ordered pairs of plus- and minus-type clique numbers
are respectively
\[
 (2,4),\qquad(4,5),\qquad(8,6),\qquad(9,8).
\]
The plus-type value reaches the ovoid bound $2^m$ exactly for
$m\le3$, in agreement with Corollary~\ref{cor:plus-clique-threshold}.
\end{remark}

\section{Cores and endomorphisms}
\label{sec:cores}

An endomorphism of a graph is a homomorphism from the graph to itself.  A
graph is a \emph{core} if all its endomorphisms are automorphisms; the core
of an arbitrary graph is its smallest retract, uniquely determined up to
isomorphism.  Roberson proved that every primitive strongly regular graph is
a pseudocore: every endomorphism is either an automorphism or a coloring
\cite{Roberson}.  More precisely, if $k$ and $s$ are its valency and least
eigenvalue, then it has a proper endomorphism if and only if
\begin{equation}\label{eq:roberson-criterion}
 \omega=1-\frac{k}{s}=\chi;
\end{equation}
in that case its core is the complete graph of that order.

The following theorem classifies the core of the graph $\Gamma_1(p)$ in even characteristic, according to its type and residual dimension.
\begin{theorem}\label{thm:core}
Let $q=2^h$ and let $G=\Gamma_1(p)$ have residual dimension $2m$ and type
$\eps$.  Then
\[
 \operatorname{core}(G)\cong
 \begin{cases}
  K_{q^m},&\eps=+1\text{ and }1\le m\le3,\\
  G,&\eps=+1\text{ and }m\ge4,\\
  G,&\eps=-1.
 \end{cases}
\]
In particular, every endomorphism in either of the last two lines is an
automorphism.  For $m\ge2$ it consequently has the affine-semilinear form
$x\mapsto Ax+b$ described in Theorem~\ref{thm:local-aut}.
\end{theorem}

\begin{proof}
First let $m\ge2$.  The parameters in
Proposition~\ref{prop:local-parameters}, together with those of the
complement in \eqref{eq:vo-v}--\eqref{eq:vo-mu}, show that $G$ is a
primitive strongly regular graph, so \eqref{eq:roberson-criterion} applies.

In plus type, Lemma~\ref{lem:clique-bounds} gives
$1-k/s=q^m$, and Theorem~\ref{thm:independent-coloring} gives
$\chi(G)=q^m$.  Corollary~\ref{cor:plus-clique-threshold} says that the
remaining equality $\omega(G)=q^m$ holds exactly when $m\le3$. Hence Roberson's criterion yields a proper endomorphism and $\operatorname{core}(G)\cong K_{q^m}$ when $2\leq m\leq 3$. For $m\ge 4$, Corollary~\ref{cor:plus-clique-threshold} gives $\omega(G) < 1-k/s$, so the equality $\omega = 1-k/s = \chi$ fails. Consequently $G$ has no proper endomorphism; it is a core. 

In minus type,
\[
 1-\frac{k}{s}=q^{m+1}-q^m+q<q^{m+1}=\chi(G),
\]
so $G$ is a core.  If $m=1$ and $\eps=-1$, then $G=K_{q^2}$. A complete graph is trivially a core, hence $\operatorname{core}(G)\cong G = K_{q^2}$. If $m=1$ and $\eps=+1$, then
$G=\overline{H(2,q)}$ has both clique and chromatic number $q$, and hence
has core $K_q$.  This completes all cases.
\end{proof}

The proper endomorphisms in the complete-core cases have a concrete
geometric shape.

\begin{corollary}
\label{cor:endomorphism-fibres}
Assume $\eps=+1$ and $1\le m\le3$.
Every nonautomorphic endomorphism of $G$ has image a clique of size $q^m$;
each of its fibres is a maximum independent set and hence an affine coset
of a maximal totally singular subspace.

Conversely, let $C$ be any clique of size $q^m$ and let $U\le W$ be any
maximal totally singular subspace.  Every coset of $U$ meets $C$ in exactly
one point, and
\begin{equation}\label{eq:explicit-retraction}
 \rho_{C,U}(x)=\text{the unique }c\in C\text{ such that }x-c\in U
\end{equation}
is a retraction of $G$ onto $C\cong K_{q^m}$.
\end{corollary}

\begin{proof}
For $2\leq m\leq 3$, the pseudocore theorem~\cite{Roberson} makes every proper endomorphism $\varphi$ a
coloring.  Its image $\operatorname{im}(\varphi)$ has at least $\chi(G)=q^m$ vertices and is a clique of
size at most $\omega(G)=q^m$, so equality holds, i.e., $|\operatorname{im}(\varphi)|=q^m$.  The fibres $\varphi^{-1}(c)$ ($c\in\operatorname{im}(\varphi)$) form a partition of the $q^{2m}$ vertices of $G$ into $q^m$ parts. Their average size is $q^{2m}/q^m = q^m$ and each fibre is an independent set. By Theorem~\ref{thm:independent-coloring}, every independent set of $G$ is contained in a coset of a totally singular subspace, and the maximum size of an independent set is $\alpha(G)=q^m$.  Consequently, each fibre must have exactly $q^m$ elements and be a maximum independent set, hence an affine coset of a maximal totally singular subspace.  

For $m=1$ ($\eps=+1$), $G = \overline{H(2,q)}$ is the complement of the $q\times q$ grid.  Its proper endomorphisms are exactly the projections onto a row or column (followed by a permutation within that row/column).  The image is a clique of size $q$, each fibre is a row or a column, which are affine cosets of the two families of maximal totally singular subspaces (the horizontal and vertical lines).  Thus the statement also holds for $m=1$.

Conversely, let $C\subseteq W$ be any clique of size $q^m$ and let $U\le W$ be any maximal totally singular subspace (hence $\dim U = m$).  Consider the set of cosets $W/U = \{U+x \mid x\in W\}$, which has size $q^{2m}/q^m = q^m$.

If two distinct points $c_1,c_2\in C$ belonged to the same coset of $U$, then $c_1-c_2\in U$, so $Q_W(c_1-c_2)=0$, contradicting the adjacency $c_1\sim c_2$ in $G$ (which requires $Q_W(c_1-c_2)\neq0$).  Hence every coset of $U$ contains at most one point of $C$.  Since there are exactly $q^m$ cosets and $|C|=q^m$, every coset meets $C$ in \emph{exactly} one point.  Thus $C$ is a transversal of the coset partition, and the map \eqref{eq:explicit-retraction} is well defined.  By construction $\rho_{C,U}(c)=c$ for every $c\in C$, so $\rho_{C,U}$ fixes $C$ pointwise.

We verify that $\rho_{C,U}$ is a graph homomorphism.  Let $x\sim y$ be adjacent in $G$, i.e.\ $Q_W(x-y)\neq0$.  Put $c_x = \rho_{C,U}(x)$ and $c_y = \rho_{C,U}(y)$.  Then $x-c_x \in U$ and $y-c_y \in U$, so
\[
(x-y) - (c_x-c_y) = (x-c_x) - (y-c_y) \in U.
\]
If $c_x = c_y$, then $x-y\in U$, implying $Q_W(x-y)=0$, contradicting $x\sim y$.  Hence $c_x \neq c_y$.  Since $C$ is a clique, $Q_W(c_x-c_y)\neq0$, so $c_x \sim c_y$ in $G$.  Therefore $\rho_{C,U}$ preserves adjacency.

We have shown that $\rho_{C,U}$ is a homomorphism from $G$ to its subgraph induced by $C$, which is isomorphic to $K_{q^m}$, and $\rho_{C,U}$ restricts to the identity on $C$.  Thus $\rho_{C,U}$ is a retraction of $G$ onto $C$.
\end{proof}

\section{Binary triple transitivity and the Terwilliger algebra}
\label{sec:terwilliger}

We finish the structural analysis with the binary Terwilliger-algebra
problem.  Let $X$ be a strongly regular graph with relations
$R_0,R_1,R_2$, where $R_0$ is equality, and let $A_0,A_1,A_2$ be their
adjacency matrices.  At a vertex $a$, let $E_{i,a}^*$ be the diagonal
projection onto the $i$th subconstituent.  The basic Terwilliger space, the
Terwilliger algebra, and the centralizer algebra of the point stabilizer are
\begin{align*}
 T_{0,a}&=\operatorname{span}_{\mathbb C}
  \{E_{i,a}^*A_jE_{k,a}^*:0\le i,j,k\le2\},\\
 T_a&=\langle A_0,A_1,A_2,E_{0,a}^*,E_{1,a}^*,E_{2,a}^*\rangle,\\
 \widetilde T_a&=\End_{\Aut(X)_a}(\mathbb C^{V(X)}).
\end{align*}
One always has $T_{0,a}\subseteq T_a\subseteq\widetilde T_a$.  Following
\cite{HermanEtAl}, $X$ is called \emph{triply transitive} when it is vertex
transitive and all three spaces are equal.

Li and Zou proved the affine-polar case by determining the relevant
stabilizer orbits \cite{LiZou}.  The following argument recovers their
result uniformly, including the two four-dimensional types, and makes
complement invariance immediate.

\begin{theorem}
\label{thm:binary-triple-transitive}
Let $m\ge2$ and $\eps\in\{+1,-1\}$.  Both
$VO_{2m}^{\eps}(2)$ and its complement are triply transitive.  In
particular, every binary first subconstituent $\Gamma_1(p)$ satisfies
\[
 T_{0,a}=T_a=\widetilde T_a
 \qquad(a\in V(\Gamma_1(p))).
\]
\end{theorem}

\begin{proof}
Take $X=VO(W,Q_W)$, it is a strongly regular graph with three relations $R_0,R_1,R_2$, where $R_0$ is equality, $R_1$ is the adjacency $x\sim y \iff Q_W(x-y)=0$ and $x\neq y$, and $R_2$ is the non‑adjacency $x\not\sim y$ and $x\neq y$ (so $Q_W(x-y)\neq0$). Fix the vertex $a=0\in W$. Since the field is $\F_2$, Theorem~\ref{thm:local-aut} gives $\Aut(X)=W\rtimes\Orth(W,Q_W)$. The point stabiliser is
\[
 \Aut(X)_0=\Orth(W,Q_W).
\]
For $i,j,k\in\{0,1,2\}$, consider the relation cell
\begin{equation}\label{eq:relation-cell}
  \mathcal{C}_{ijk}=\{(x,y)\in W\times W : (0,x)\in R_i,\;(x,y)\in R_j,\;(y,0)\in R_k\}.
\end{equation}
We claim that every nonempty $\mathcal{C}_{ijk}$ is a single orbit under the action of $\Orth(W,Q_W)$ on pairs.

Indeed, take $(x,y)$ and $(x',y')$ in the same cell. The three conditions determine:
\begin{itemize}
  \item whether $x=0$ ($i=0$), $Q_W(x)=0$ and $x\neq0$ ($i=1$), or $Q_W(x)\neq0$ ($i=2$);
  \item whether $y=0$ ($k=0$), $Q_W(y)=0$ and $y\neq0$ ($k=1$), or $Q_W(y)\neq0$ ($k=2$);
  \item the relation between $x$ and $y$: $x=y$ ($j=0$), $Q_W(x-y)=0$ and $x\neq y$ ($j=1$), or $Q_W(x-y)\neq0$ ($j=2$).
\end{itemize}
Since the only nonzero element of $\mathbb{F}_2$ is $1$, the values $Q_W(x),Q_W(y),Q_W(x+y)$ are completely determined by these data. Moreover, from $Q_W(x+y)=Q_W(x)+Q_W(y)+B_W(x,y)$ we recover the polar value $B_W(x,y)$, which is also either $0$ or $1$.  

Consequently, the assignment $x\mapsto x'$, $y\mapsto y'$ preserves $Q_W$ and $B_W$ on the sets $\{x,y\}$ and $\{x',y'\}$. It therefore extends linearly to an isometry between the subspaces $\langle x,y\rangle$ and $\langle x',y'\rangle$ of $W$. Since the
polar form of $Q_W$ is nondegenerate, Witt's extension theorem in characteristic two extends it to an element  $g\in\Orth(W,Q_W)$(see~\cite{Wan}). Clearly $g(x)=x'$ and $g(y)=y'$, so $(x,y)$ and $(x',y')$ belong to the same $\Orth(W,Q_W)$-orbit. 

The group $\Orth(W,Q_W)$ acts on $W$ by $x\mapsto g(x)$. This action induces an action on $W\times W$ via $g(x,y)=(g(x),g(y))$. The orbits of this action are exactly the nonempty sets $\mathcal{C}_{ijk}$, because every pair $(x,y)$ belongs to some $\mathcal{C}_{ijk}$, and the action clearly preserves each cell (the relations $(0,x)\in R_i$, $(x,y)\in R_j$, $(y,0)\in R_k$ are invariant under an orthogonal transformation fixing $0$). This proves the claim.

For each orbit $\mathcal{O}$ of $\Orth(W,Q_W)$ on $W\times W$, let $M_{\mathcal{O}}\in\mathbb{C}^{W\times W}$ be its characteristic matrix:
\[
(M_{\mathcal{O}})_{x,y}= \begin{cases} 1 & \text{if }(x,y)\in\mathcal{O},\\ 0 & \text{otherwise.} \end{cases}
\]
These matrices are linearly independent because the orbits partition $W\times W$, so each position $(x,y)$ is non‑zero in exactly one $M_{\mathcal{O}}$. Hence any non‑trivial linear combination $\sum_{\mathcal{O}} c_{\mathcal{O}} M_{\mathcal{O}}=0$ would force all coefficients $c_{\mathcal{O}}$ to be zero.

Let $\mathbb{C}^W$ be the complex vector space of formal linear combinations of the vertices. A natural basis of this space is $\{\mathbf{e}_x \mid x\in W\}$, where $\mathbf{e}_x$ is the vector with $1$ in the $x$‑coordinate and $0$ elsewhere. We \emph{lift} the action of $G$ on $W$ to a linear action on $\mathbb{C}^W$ by defining, for each $g\in G$, a linear map $P(g)\colon\mathbb{C}^W\to\mathbb{C}^W$ whose effect on the basis is dictated by the original permutation:
\[
P(g)\,\mathbf{e}_y = \mathbf{e}_{g(y)} \qquad \forall\,y\in W.
\]
Since a linear map is uniquely determined by its values on a basis, this definition makes $P(g)$ a well‑defined invertible linear operator. Explicitly, for an arbitrary vector $v = \sum_{y\in W} v_y \,\mathbf{e}_y \in \mathbb{C}^W$,
\[
P(g)\,v = \sum_{y\in W} v_y \,\mathbf{e}_{g(y)}.
\]

The matrix of $P(g)$ with respect to the standard basis $\{\mathbf{e}_x\}$ is obtained by reading off the coefficients of $P(g)\,\mathbf{e}_y$:
\[
(P(g))_{x,y}= \begin{cases} 1 & \text{if }g(y)=x,\\ 0 & \text{otherwise.} \end{cases}
\]
Thus $P(g)$ is a permutation matrix, and the assignment $g\mapsto P(g)$ is a group homomorphism from $G$ into the group of $|W|\times|W|$ permutation matrices. This is the permutation representation of $G$ on $\mathbb{C}^W$.

The centraliser algebra of this permutation group is
\[
\widetilde T_0 = \End_{\Orth(W,Q_W)}(\mathbb{C}^W) = \{ M\in\mathbb{C}^{W\times W} \mid M P(g) = P(g) M \text{ for all } g\in\Orth(W,Q_W) \}.
\]

We claim that a matrix $M$ belongs to $\widetilde T_0$ if and only if $M$ is constant on each orbit of $\Orth(W,Q_W)$ on $W\times W$, i.e.\ $M_{x,y}=M_{x',y'}$ whenever $(x,y)$ and $(x',y')$ lie in the same orbit.

Indeed, the condition $M P(g) = P(g) M$ is equivalent to $M = P(g) M P(g)^{-1}$. Conjugating by the permutation matrix $P(g)$ permutes the entries of $M$ according to the action of $g$ on both coordinates:
\[
\begin{aligned}
\bigl(P(g) M P(g)^{-1}\bigr)_{x,y}
&= \bigl(P(g) M P(g^{-1})\bigr)_{x,y} \\
&= \sum_{u,v\in W} P(g)_{x,u} \, M_{u,v} \, P(g^{-1})_{v,y} \\
&= \sum_{u,v\in W} \delta_{x,g(u)} \, M_{u,v} \, \delta_{v,g^{-1}(y)} \\
&= M_{g^{-1}(x),\, g^{-1}(y)} .
\end{aligned}
\]
Thus $M$ commutes with all $P(g)$ precisely when $M_{x,y}=M_{g(x),g(y)}$ for all $g\in\Orth(W,Q_W)$ and all $x,y\in W$, which means $M$ is constant on each orbit of the group acting on $W\times W$ by $g(x,y)=(g(x),g(y))$.

Therefore any $M\in\widetilde T_0$ can be written uniquely as
\[
M = \sum_{\mathcal{O}} c_{\mathcal{O}} M_{\mathcal{O}},
\]
where $\mathcal{O}$ runs over all orbits on $W\times W$ and $c_{\mathcal{O}}\in\mathbb{C}$ is the common value of $M$ on $\mathcal{O}$. Conversely, any such linear combination clearly commutes with every $P(g)$ because each $M_{\mathcal{O}}$ does. Hence
\[
\widetilde T_0 = \operatorname{span}\{ M_{\mathcal{O}} : \mathcal{O} \text{ is an orbit of } \Orth(W,Q_W) \text{ on } W\times W \}.
\]

Now consider the matrices $E_{i,0}^*A_jE_{k,0}^*$ that generate $T_{0,0}$. Recall that the relations $R_0,R_1,R_2$ partition $W\times W$. At the fixed base vertex $0$, the $i$‑th subconstituent ($i=0,1,2$) is the set of vertices $x\in W$ such that $(0,x)\in R_i$.
The diagonal projection $E_{i,0}^*$ onto the $i$‑th subconstituent is defined by its action on basis vectors:
\[
E_{i,0}^*\,\mathbf{e}_x = 
\begin{cases}
\mathbf{e}_x, & \text{if } (0,x)\in R_i,\\[2pt]
0,             & \text{otherwise}.
\end{cases}
\]
Equivalently, as a matrix,
\[
(E_{i,0}^*)_{x,x}= \begin{cases} 1 & \text{if }(0,x)\in R_i,\\ 0 & \text{otherwise,} \end{cases}
\qquad\text{and}\qquad
(E_{i,0}^*)_{x,y}=0\;\;(x\neq y).
\]

The adjacency matrix $A_j$ of relation $R_j$ is defined by
\[
(A_j)_{x,y}= \begin{cases} 1 & \text{if }(x,y)\in R_j,\\ 0 & \text{otherwise.} \end{cases}
\]

Now compute the $(x,y)$-entry of the product $E_{i,0}^*A_jE_{k,0}^*$. Using the standard matrix multiplication formula,
\[
\begin{aligned}
\bigl(E_{i,0}^*A_jE_{k,0}^*\bigr)_{x,y}
&= \sum_{u,v\in W} (E_{i,0}^*)_{x,u}\, (A_j)_{u,v}\, (E_{k,0}^*)_{v,y} \\
&= \sum_{u,v} (E_{i,0}^*)_{x,x}\delta_{x,u}\, (A_j)_{u,v}\, (E_{k,0}^*)_{y,y}\delta_{v,y} \quad\text{(since }E_{i,0}^*\text{ is diagonal)}\\
&= (E_{i,0}^*)_{x,x}\, (A_j)_{x,y}\, (E_{k,0}^*)_{y,y}.
\end{aligned}
\]

Thus
\[
\bigl(E_{i,0}^*A_jE_{k,0}^*\bigr)_{x,y} = 
\begin{cases}
1 & \text{if } (E_{i,0}^*)_{x,x}=1,\; (A_j)_{x,y}=1,\; (E_{k,0}^*)_{y,y}=1,\\
0 & \text{otherwise}.
\end{cases}
\]
By definition, this is equivalent to
\[
(0,x)\in R_i,\quad (x,y)\in R_j,\quad (y,0)\in R_k,
\]
which is precisely the defining condition of the cell $\mathcal{C}_{ijk}$ given in~\eqref{eq:relation-cell}. Therefore
\[
E_{i,0}^*A_jE_{k,0}^* = M_{\mathcal{C}_{ijk}},
\]
the characteristic matrix of the cell $\mathcal{C}_{ijk}$.

Since the nonempty cells $\mathcal{C}_{ijk}$ are the orbits of $\Orth(W,Q_W)$ on $W\times W$, the matrices $E_{i,0}^*A_jE_{k,0}^*$ are precisely the orbital matrices of the point stabiliser. Their span therefore equals the centraliser algebra $\widetilde T_0$:
\[
\widetilde T_0 = \operatorname{span}\{ E_{i,0}^*A_jE_{k,0}^* : i,j,k\in\{0,1,2\} \} = T_{0,0}.
\]
The inclusions $T_{0,0}\subseteq T_0\subseteq\widetilde T_0$ always hold, so we have equality throughout:
\[
T_{0,0}=T_0=\widetilde T_0 .
\]

The whole automorphism group $\Aut(X) = W \rtimes \Orth(W,Q_W)$ acts transitively on the vertex set $W$. For an arbitrary vertex $a\in W$, choose the translation $t_{-a}\colon x\mapsto x-a$, which is an automorphism of $X$ sending $a$ to $0$.  Let $P(t_{-a})$ be the permutation matrix of $t_{-a}$, defined as before.  
Conjugating the generating matrices at $0$ by $P(t_{-a})$ yields the corresponding generators at $a$:
\begin{itemize}
\item For the dual projections, one checks directly that $P(t_{-a})E_{i,0}^*P(t_{-a})^{-1}=E_{i,a}^*$, because $(0,x)\in R_i$ if and only if $(a,\,t_{-a}(x))\in R_i$.
\item The adjacency matrices satisfy $P(t_{-a})A_jP(t_{-a})^{-1}=A_j$, since any translation is a graph automorphism and therefore preserves the relations $R_j$.
\end{itemize}
Consequently, the whole Terwilliger algebra and the basic Terwilliger space are conjugated:
\[
T_{0,a}=P(t_{-a})\,T_{0,0}\,P(t_{-a})^{-1},\qquad
T_a=P(t_{-a})\,T_0\,P(t_{-a})^{-1}.
\]
For the centraliser algebra, the stabiliser at $a$ is $\Aut(X)_a = t_{-a}\,\Aut(X)_0\,t_{-a}^{-1}$, and passing to the permutation matrices gives
\[
\widetilde T_a = \End_{\Aut(X)_a}(\mathbb{C}^W)
               = P(t_{-a})\,\End_{\Aut(X)_0}(\mathbb{C}^W)\,P(t_{-a})^{-1}
               = P(t_{-a})\,\widetilde T_0\,P(t_{-a})^{-1}.
\]
We have already proved $T_{0,0}=T_0=\widetilde T_0$.  Applying the conjugation by $P(t_{-a})$ to this chain of equalities, we obtain
\[
T_{0,a}=T_a=\widetilde T_a \qquad\text{for every } a\in W.
\]
Thus all three spaces coincide at every vertex, which together with vertex transitivity means that $X$ is triply transitive in the sense of Terwilliger.

Finally, the complement graph $\overline{X} = \overline{VO(W,Q_W)}$ interchanges $R_1$ and $R_2$ but leaves $R_0$ unchanged. Its automorphism group is identical to that of $X$, and the three Terwilliger spaces depend only on the permutation group and the partition of vertices into subconstituents, which is also unchanged. Hence the same
conclusion holds for $\overline X=\Gamma_1(p)$.
\end{proof}

\begin{remark}
Theorem~\ref{thm:binary-triple-transitive} should be read as a streamlined
derivation and a local-complement corollary of the Li--Zou theorem, not as a
new resolution of Conjecture~6.17 of \cite{HermanEtAl}.  Its advantage in
the present setting is that the affine chart makes the complete Gram data
visible from the three graph relations, so no separate low-dimensional
computer calculation is needed.
\end{remark}

\section{Remaining problems}
\label{sec:problems}

The results above close the binary Terwilliger problem and the core problem,
but they leave several concrete extremal classification questions.

\begin{enumerate}
\item Determine $\omega(\overline{VO_{2m}^{\eps}(q)})$ for general
$q=2^h>2$ in minus type, and in plus type when $m\ge4$.  The present paper
gives the exact answer for the infinite binary subfamily, while
\eqref{eq:plus-clique-bound}--\eqref{eq:minus-clique-bound} and
Corollary~\ref{cor:plus-clique-threshold} give the general bounds and the
strictness of the plus-type Delsarte bound.

\item Classify the maximum cliques when the plus-type Delsarte bound is
attained.  By Theorem~\ref{thm:clique-ovoid}, for residual rank three this
is precisely the classification of ovoids of $Q^+(7,2^h)$ through a fixed
point.  The low-degree results of \cite{BartoliEtAl} are substantial, but
do not give a classification of arbitrary ovoids.

\item Determine the full endomorphism monoid in the complete-core cases.
Corollary~\ref{cor:endomorphism-fibres} determines every fibre individually
and constructs a large family of retractions.  What remains is to classify
the partitions of $W$ into affine maximal singular flats that can occur as
the fibre partition of an endomorphism, together with the compatible
ovoidal image clique.
\end{enumerate}

\section{Concluding remarks}

The identity $B(v_x,v_y)=Q_W(x-y)$ is the essential observation.  It
simultaneously explains the Cayley structure, the strongly regular
parameters, the translation subgroup, and the appearance of the
semisimilarity group.  Once the neighbourhood is recognized as the
complement of an affine polar graph, the long coordinate reconstruction of
automorphisms is replaced by a standard exact automorphism theorem and two
explicit extension formulas.

The same identity also controls extremal sets.  It turns independent sets
into affine singular flats and Delsarte cliques into ovoids, and in the
binary case reduces every clique to a simplex-type Gram matrix.  These
facts, combined with the general pseudocore theorem, determine all local
cores and the shape of every proper endomorphism in the complete-core
range.  Over $\F_2$, the three graph relations recover the full Gram data
of two vectors, explaining the exceptional triple transitivity without a
case-by-case orbit computation.

The reformulation therefore separates three different phenomena cleanly:
residual dimension at least four is uniform for automorphisms; residual
planes produce the only extra symmetric-group automorphisms; and binary
fields produce the additional Gram rigidity behind the exact clique formula
and Terwilliger equality.

\section*{Acknowledgment}


\section*{Data availability}

No datasets were generated or analysed in this study.

\section*{Declaration of competing interest}

The authors declare that they have no known competing financial interests
or personal relationships that could have influenced the work reported in
this paper.

\end{document}